%% file: final_arxiv.tex
\documentclass[11pt]{amsart}

\input{style}

\usepackage[margin=1.15in]{geometry}
\usepackage{pgfplots}
\pgfplotsset{compat=1.18}
\usepgfplotslibrary{fillbetween}
\usetikzlibrary{arrows.meta,calc,decorations.pathmorphing,positioning,shapes.geometric,backgrounds}

\usetikzlibrary{fit}
\definecolor{figteal}{RGB}{0,128,128}
\definecolor{figorange}{RGB}{217,119,6}
\definecolor{figgray}{RGB}{120,120,120}
\definecolor{figred}{RGB}{197,52,46}
\pgfplotsset{
	figaxis/.style={
		axis line style={gray!55},
		tick style={gray!55},
		label style={font=\small},
		tick label style={font=\footnotesize},
		title style={font=\small,figteal!75!black},
		legend cell align={left},
		legend style={font=\footnotesize,draw=gray!45,rounded corners=1pt,
			fill=white,fill opacity=0.92,text opacity=1},
		grid style={gray!18,line width=0.3pt},
	},
}
\tikzset{
	blob/.style={circle,draw=figteal!85!black,fill=figteal!16,thick,
		minimum size=15pt,inner sep=1pt},
	rootblob/.style={circle,draw=figteal!88!black,fill=figteal!30,very thick,
		minimum size=20pt,inner sep=1pt},
	decoration node/.style={circle,draw=figorange!85!black,fill=figorange!16,
		thick,inner sep=0pt},
	shortcut/.style={figorange!88!black,densely dashed,thick},
	treeedge/.style={figteal!75!black,thick},
	callout/.style={draw=gray!60,fill=white,rounded corners=2pt,
		align=center,inner sep=3pt,font=\footnotesize},
	paneltitle/.style={figteal!75!black,font=\small},
	faded/.style={draw=gray!55,fill=gray!12,text=gray!70},
	badedge/.style={figred!90!black,densely dashed,very thick},
	stagebox/.style={draw=gray!60,fill=gray!5,rounded corners=2pt,
		thick,align=center,inner sep=5pt,font=\small},
	panel/.style={draw=gray!40,fill=gray!4,rounded corners=3pt,
		thick,inner sep=0pt},
	flowarrow/.style={-{Latex[length=2.6mm]},thick,gray!75!black},
}

\title[Subcritical percolation and network archaeology]
{Subcritical percolation and network archaeology on random recursive tree substrate networks}
\author{Shankar Bhamidi}
\author{Akshay Sakanaveeti}
\date{July 23, 2026}

\newcommand{\PP}{\mathbb P}
\newcommand{\EE}{\mathbb E}
\newcommand{\Var}{\operatorname{Var}}

\newcommand{\Poi}{\operatorname{Poi}}
\newcommand{\rank}{\operatorname{rank}}

\newcommand{\ErdosRenyi}{\operatorname{ER}}

\newcommand{\full}{\operatorname{full}}

\newcommand{\Root}{\operatorname{root}}

\newcommand{\TopJ}{\operatorname{TopJ}}
\newcommand{\Comp}{\operatorname{Comp}}

\begin{document}
	
	\begin{abstract}
		We study a network-archaeology problem for a dynamic graph whose latent
		substrate is a random recursive tree and whose observed topology is enriched
		by an independent homogeneous Erd\H{o}s--R\'enyi shortcut layer.  From a single
		unlabeled snapshot, the goal is to construct a confidence set of deterministic
		size for the first vertex.  Since shortcut edges create cycles, the usual
		tree-based arguments using Jordan centrality do not apply directly.  Our method
		uses auxiliary subcritical bond percolation to expose a tree-like
		renormalized structure: retained recursive-tree clusters form heavy-tailed
		blobs, retained shortcuts connect these blobs through a subcritical rank-one
		random graph, and large components are leading backbone blobs decorated by
		subcritical shortcut pieces.  Applying Jordan centrality inside the largest
		auxiliary percolation components then gives a deterministic-size root
		confidence set for the cyclic observed network.
	\end{abstract}
	
	\maketitle
	
	\section{Introduction}
	\label{sec:introduction}
	
	Many large networks are temporal objects observed only after a long period of
	growth.  Vertex labels, time stamps, and early states may be unavailable, while
	the present-day topology can still retain partial information about the order
	in which the network was formed.  Network archaeology asks how much of this
	latent history can be recovered from a single static snapshot; see, for
	example, the algorithmic and statistical formulations in
	\cite{NavlakhaKingsford2011NetworkArchaeology,YoungStOngeLaurenceMurphyHebertDufresneDesrosiers2019History,CraneXu2020History}.
	A basic probabilistic version is the root-finding problem: given an unlabeled
	graph generated by a growth model, output a vertex set that contains the first
	vertex with high probability and whose size depends on an error tolerance but
	not on the final graph size.
	
	The benchmark theory is for random growing trees.  Deterministic-size
	confidence sets for uniform and preferential attachment trees were established
	in \cite{BubeckDevroyeLugosi2017FindingAdam}.  One of the central statistics
	in this theory is Jordan centrality, or centroid centrality: vertices are
	ranked according to the size of the largest component left after deleting
	them.  Persistence results for this statistic were proved in
	\cite{JogLoh2018}, while root-finding and persistence theorems for broad
	classes of attachment-function-driven growing trees were developed in
	\cite{BanerjeeBhamidi2022JordanRoot} using Crump--Mode--Jagers
	branching-process embeddings.
	
	The tree setting is special in a structural sense.  Removing a vertex from a
	tree exposes a clean family of descendant components, and centrality
	comparisons can be reduced to estimates for subtree sizes.  In a genuine
	network, even a sparse layer of cycles and long-range edges destroys this
	decomposition.  Several recent works have begun to treat non-tree or
	data-noisy versions of network archaeology, including latent growth plus
	Erd\H{o}s--R\'enyi noise and random recursive directed or Cooper--Frieze
	networks~\cite{CraneXu2021LatentGrowth,BriendCalvilloLugosi2022RecursiveDAGs}.
	Very recent work~\cite{DevroyeLugosiMaitra2026NoisyTrees} obtains root-finding
	guarantees for the same model over a broader noise regime; we compare the two
	approaches in Section~\ref{sec:discussion}.
	
	We study a sparse dynamic network model with an explicit random recursive tree
	substrate.  Let \(T_n\) be the random recursive tree on \([n]\), rooted at
	\(1\), and let \(H_n\) be an independent Erd\H{o}s--R\'enyi shortcut graph with
	edge probability \(\gl/n\).  The observed graph is \(G_n=T_n\cup H_n\),
	and the input to the root-finding algorithm is only its unlabeled topology.
	The recursive tree records a latent age order, while the shortcut layer adds
	homogeneous non-genealogical edges.  These shortcuts are sparse enough to
	retain the large-scale imprint of growth, but they hide the genealogical
	decomposition on which the usual tree arguments rely.
	
	Our main device is auxiliary subcritical percolation.  Independently retain
	each edge of \(G_n\); we denote the retention parameter by \(p\) in the
	structural analysis and by \(q\) in the root-finding procedure.
	The retained tree edges partition the latent recursive tree into connected
	clusters, or blobs.  These blobs are heavy-tailed, and the root blob has tight
	rank among the leading blobs.  Conditional on the blob sizes, the retained
	shortcuts induce a rank-one inhomogeneous random graph whose weights are the
	blob sizes.  Below the rank-one critical curve, a leading blob is not absorbed
	into a giant shortcut component; instead it is amplified by a subcritical
	cloud of shortcut decorations.  The effective branching factor in the
	recursive-tree case is
	\[
	\rho(p,\gl)=\frac{p\gl}{1-2p},
	\]
	so the structural subcritical condition is \(p<1/(\gl+2)\).  This
	renormalized description converts the cyclic network back into a decorated
	tree-like object at the scales needed for root finding.
	Figure~\ref{fig:proof-strategy-overview} gives a schematic view of this
	thinning-and-ranking pipeline.
	
	\begin{figure}[tbp]
		\centering
		\begin{tikzpicture}[
			x=1.07cm,y=1.07cm,
			every node/.style={font=\small}
			]
			\path[panel] (-1.25,-0.9) rectangle (1.25,0.9);
			\node[figteal!70!black] at (0,1.13) {observed graph};
			\coordinate (o1) at (-0.78,-0.28);
			\coordinate (o2) at (-0.23,0.36);
			\coordinate (o3) at (0.42,0.16);
			\coordinate (o4) at (0.82,-0.43);
			\coordinate (o5) at (-0.08,-0.60);
			\draw[treeedge] (o1)--(o2)--(o3)--(o4);
			\draw[treeedge] (o2)--(o5);
			\draw[shortcut] (o1) to[bend right=14] (o4);
			\draw[shortcut] (o5) to[bend left=12] (o3);
			\foreach \p in {o1,o2,o3,o4,o5} {
				\node[blob,minimum size=9pt] at (\p) {};
			}
			
			\path[panel] (2.75,-0.9) rectangle (5.25,0.9);
			\node[figteal!70!black] at (4.0,1.13) {\(q\)-percolated blobs};
			\draw[treeedge,line width=1.1pt] (3.25,0.32)--(3.62,0.0)--(3.36,-0.36);
			\foreach \x/\y/\s in {3.25/0.32/9pt,3.62/0.0/12pt,3.36/-0.36/8pt} {
				\node[blob,minimum size=\s] at (\x,\y) {};
			}
			\draw[treeedge,line width=1.1pt] (4.35,0.36)--(4.68,0.02);
			\foreach \x/\y/\s in {4.35/0.36/13pt,4.68/0.02/9pt} {
				\node[rootblob,minimum size=\s] at (\x,\y) {};
			}
			\node[blob,minimum size=10pt] at (4.15,-0.55) {};
			\draw[shortcut] (3.62,0.0) to[bend left=12] (4.35,0.36);
			\draw[shortcut] (4.68,0.02) to[bend left=10] (4.15,-0.55);
			
			\path[panel] (6.95,-0.9) rectangle (9.45,0.9);
			\node[figteal!70!black] at (8.2,1.13) {rank and search};
			\fill[figteal!10,rounded corners=1pt] (7.43,0.05) rectangle (8.96,0.62);
			\foreach \y/\w/\c in {0.48/1.28/teal,0.16/1.00/orange,-0.16/0.76/teal,-0.48/0.55/gray} {
				\draw[draw=\c!70!black,fill=\c!15,rounded corners=1pt,thick]
				(7.58,\y) rectangle ++(\w,0.17);
			}
			\node[figteal!70!black] at (8.23,-0.72) {\(\TopJ_L\) in top \(K\)};
			
			\draw[-{Latex[length=2.6mm]},thick] (1.42,0) -- node[above] {thin} (2.58,0);
			\draw[-{Latex[length=2.6mm]},thick] (5.42,0) -- node[above] {rank} (6.78,0);
		\end{tikzpicture}
		\caption{Schematic overview of the root-finding strategy.  Left: the observed
			graph; the solid and dashed edge styles indicate the latent tree and shortcut
			layers, a distinction unavailable to the algorithm.  Middle: after auxiliary
			\(q\)-percolation, the analysis decomposes the retained graph into tree blobs
			joined by retained shortcut edges.  Right: the resulting components are ranked
			by size, and Jordan centrality is applied within the leading components.}
		\label{fig:proof-strategy-overview}
	\end{figure}
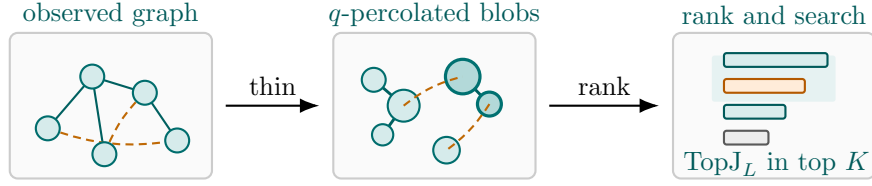
	
	\subsection{Informal description of our contributions}
	The paper has two main conclusions.  First, we prove a subcritical component
	picture for the percolated shortcut model at any retention parameter
	\(p<1/(\gl+2)\): the largest full components have, to first order, the sizes of
	the largest retained recursive-tree blobs multiplied by the factor
	\((1-\rho(p,\gl))^{-1}\).  This is analogous to the subcritical power-law
	component theorem for the configuration model~\cite{Janson2008SubcriticalPowerLaw},
	where the largest subcritical cluster sizes exceed the maximal degree by a
	similar factor; here, however, the rank-one weights are generated dynamically
	by recursive-tree percolation.  Second, for
	\(q<\min\{1/(\gl+2),1/3\}\), we use this component picture as the starting point
	for a root-finding algorithm.  The root component appears among the largest
	\(q\)-percolated components; inside it, the root blob is a uniform attachment
	tree carrying shortcut decorations with controlled moments.  Weighted Jordan
	centrality on this skeleton, followed by a deterministic comparison with
	ordinary graph Jordan centrality on the decorated component, yields a
	deterministic-size confidence set for the root.
	
	We treat the random recursive tree here because the required probabilistic
	inputs are explicit.  The Yule embedding with neutral mutations gives the
	Yule--Simon empirical blob law, the largest blobs have order \(n^p\), the
	susceptibility is \((1-2p)^{-1}\), and the shortcut subcritical condition is
	\(p<1/(\gl+2)\).  Conditional on its size, the root blob is a uniform attachment
	tree, so weighted Jordan centrality can be analyzed directly and then
	transferred to graph Jordan centrality in the decorated component.  In ongoing
	work, we extend this program to preferential-attachment tree substrates with a
	general degree-based attachment function \(f\).  In that setting, the Yule
	inputs are replaced by the corresponding Crump--Mode--Jagers inputs:
	Malthusian exponents for blob growth, exponential-age mixtures for the
	empirical blob law, susceptibility constants for the rank-one shortcut graph,
	and a weighted Jordan theorem for the thinned genealogy.

	\paragraph{\bf Organization of the paper.}
	Section~\ref{sec:model-notation} defines the model and the auxiliary
	root-finding algorithm.  Section~\ref{sec:results} states the two main
	results.  Section~\ref{sec:discussion} discusses related work, limitations,
	and possible extensions to broader attachment substrates.  The proofs are
	given in Section~\ref{sec:proofs}.
	
	\section{Model and notation}
	\label{sec:model-notation}
	
	Let \(T_n\) be the random recursive tree on \([n]\), rooted at \(1\).  Recall
	that a recursive tree on \([n]\) is a rooted labeled tree in which labels
	increase along every path away from the root.  Equivalently, \(T_n\) can be
	constructed dynamically by starting from vertex \(1\) and, for each
	\(m=2,\ldots,n\), attaching the new vertex \(m\) to a uniformly chosen vertex
	of \([m-1]\); the resulting distribution is uniform over recursive trees on
	\([n]\).  Random recursive trees are a standard uniform-attachment model for
	recursive structures and growing networks; see
	\cite{SmytheMahmoud1995Survey,Drmota2009RandomTrees}.  In network archaeology
	this is a useful null substrate because new vertices have no degree, fitness,
	or spatial preference.  Thus root recovery in the random recursive tree is
	already somewhat surprising: even when every newcomer attaches uniformly at
	random and the final labels are erased, the topology still retains enough
	age-asymmetry to localize the first vertex with deterministic-size confidence
	sets~\cite{BubeckDevroyeLugosi2017FindingAdam}.  The question here is whether
	that signal persists after adding shortcut edges.

	\definecolor{treecol}{RGB}{15,110,86}     
	\definecolor{shortcol}{RGB}{216,90,48}    
	\definecolor{rootcol}{RGB}{15,110,86}
	\definecolor{blobfill}{RGB}{225,245,238}
	\definecolor{compcol}{RGB}{83,74,183}
	
	\newcommand{\vertices}{%
		\coordinate (v1)  at (3.4,3.3);
		\coordinate (v2)  at (2.0,2.6);
		\coordinate (v3)  at (4.8,2.6);
		\coordinate (v4)  at (1.2,1.8);
		\coordinate (v5)  at (2.7,1.8);
		\coordinate (v6)  at (4.2,1.8);
		\coordinate (v7)  at (5.5,1.8);
		\coordinate (v8)  at (0.8,1.0);
		\coordinate (v9)  at (1.8,1.0);
		\coordinate (v10) at (3.0,1.0);
		\coordinate (v11) at (4.7,1.0);
		\coordinate (v12) at (5.9,1.0);}
	\newcommand{\verticesB}{%
		\coordinate (v1)  at (3.4,3.3);
		\coordinate (v2)  at (2.0,2.6);
		\coordinate (v3)  at (4.8,2.6);
		\coordinate (v4)  at (1.2,1.8);
		\coordinate (v5)  at (2.7,1.8);
		\coordinate (v6)  at (4.2,1.8);
		\coordinate (v7)  at (5.9,1.9);
		\coordinate (v8)  at (5.35,0.55);
		\coordinate (v9)  at (1.8,1.0);
		\coordinate (v10) at (3.0,1.0);
		\coordinate (v11) at (4.55,1.0);
		\coordinate (v12) at (6.35,1.05);}
	\newcommand{\drawnodes}{%
		\foreach \i in {2,...,12}
		\node[circle,draw=black!70,fill=white,inner sep=1pt,minimum size=4.0mm,font=\scriptsize] (n\i) at (v\i) {\i};
		\node[circle,draw=rootcol,very thick,fill=rootcol!15,inner sep=1pt,minimum size=4.4mm,font=\scriptsize\bfseries] (n1) at (v1) {1};}
	\newcommand{\treeedges}[1]{\foreach \a/\b in {1/2,1/3,2/4,2/5,3/6,3/7,4/8,4/9,5/10,6/11,7/12} \draw[#1] (v\a) -- (v\b);}
	\newcommand{\shortcutedges}[1]{\foreach \a/\b in {5/6,9/10,8/11,2/12} \draw[#1] (v\a) -- (v\b);}
	\newcommand{\alltree}[1]{\treeedges{#1}}
	\newcommand{\allshort}[1]{\shortcutedges{#1}}
	\newcommand{\retainedtree}[1]{\foreach \a/\b in {1/2,2/5,5/10,3/6,6/11,4/9,7/12} \draw[#1] (v\a) -- (v\b);}
	\newcommand{\retainedshort}[1]{\foreach \a/\b in {9/10,8/11} \draw[#1] (v\a) -- (v\b);}
	\newcommand{\deletededges}[1]{\foreach \a/\b in {1/3,2/4,3/7,4/8,5/6,2/12} \draw[#1] (v\a) -- (v\b);}
	\newcommand{\blobhulls}[1]{%
		\begin{scope}[on background layer]
			\node[ellipse,fill=blobfill,draw=rootcol,thick,inner sep=-2pt,rotate=51,fit=(v1)(v2)(v5)(v10)] {};
			\node[ellipse,fill=black!6,draw=black!45,inner sep=-3pt,rotate=62,fit=(v3)(v6)(v11)] {};
			\node[ellipse,fill=black!6,draw=black!45,inner sep=-1pt,rotate=48,fit=(v4)(v9)] {};
			\node[ellipse,fill=black!6,draw=black!45,inner sep=-2pt,rotate=55,fit=(v7)(v12)] {};
			\node[ellipse,fill=black!6,draw=black!45,inner sep=1.6pt,fit=(v8)] {};
			#1
	\end{scope}}

	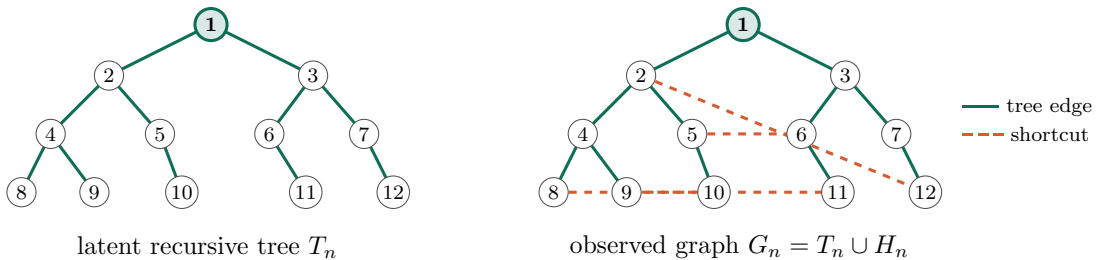
\begin{figure}[b]
		\centering
		\resizebox{0.92\textwidth}{!}{%
			\begin{tikzpicture}[scale=1.0]
				\begin{scope}[xshift=0cm]
					\vertices
					\treeedges{treecol,very thick}
					\drawnodes
					\node[font=\small] at (3.35,0.25) {latent recursive tree $T_n$};
				\end{scope}
				\begin{scope}[xshift=7.3cm]
					\vertices
					\treeedges{treecol,very thick}
					\shortcutedges{shortcol,very thick,dashed}
					\drawnodes
					\node[font=\small] at (3.35,0.25) {observed graph $G_n=T_n\cup H_n$};
				\end{scope}
				\node[font=\scriptsize,align=left] at (14.6,2.0)
				{\textcolor{treecol}{\rule[0.5ex]{5mm}{1.1pt}} tree edge\\[1pt]
					\textcolor{shortcol}{\rule[0.5ex]{1.4mm}{1.1pt}\hspace{0.7mm}\rule[0.5ex]{1.4mm}{1.1pt}\hspace{0.7mm}\rule[0.5ex]{1.4mm}{1.1pt}} shortcut};
			\end{tikzpicture}%
		}
		\caption{The model. Left: the latent random recursive tree $T_n$, rooted at
			vertex $1$. Right: the observed graph $G_n=T_n\cup H_n$, where
			$H_n\sim\mathrm{ER}(n,\lambda/n)$ is the independent shortcut layer (dashed).
			In the unlabeled snapshot $G_n^\circ$, tree edges and shortcut edges are
			indistinguishable; the colors and vertex labels shown here are not
			available to the observer.}
		\label{fig:model-layers}
	\end{figure}
	
	Let \(H_n\) be an independent Erd\H{o}s--R\'enyi graph on \([n]\) with edge
	probability \(\gl/n\). Write
	\[
	G_n=T_n\cup H_n .
	\]
	Let \(G_n^\circ\) denote the unlabeled isomorphism class of \(G_n\).  The
	network archaeology, or root-reconstruction, problem is to identify the
	original vertex \(1\) from \(G_n^\circ\) alone: the algorithm is not given the
	labels \([n]\), the birth times, the decomposition \(G_n=T_n\cup H_n\), or a
	mark distinguishing the root.  Figure~\ref{fig:model-layers} illustrates the
	latent two-layer construction.

	Following the confidence-set formulation of
	root finding in growing random trees
	\cite{BubeckDevroyeLugosi2017FindingAdam,BanerjeeBhamidi2022JordanRoot}, a
	budget-\(M\) reconstruction rule is an isomorphism-invariant, possibly
	randomized, map \(\mathcal A_M\) that returns at most \(M\) vertices of the
	unlabeled input graph.  For an error tolerance \(\eps\in(0,1)\), such a rule
	succeeds if
	\[
	\liminf_{n\to\infty}
	\PP\{1\in \mathcal A_M(G_n^\circ)\}
	\ge 1-\eps .
	\]
	The point is that \(M\) may depend on \(\eps\) and on the model parameters,
	but not on \(n\).  The theorem below constructs such a deterministic-size
	confidence set by adding auxiliary percolation randomness after the unlabeled
	graph has been observed.  To simplify notation in the proofs, particularly in
	Section~\ref{sec:proof-network-archaeology}, we write \(G_n\) for the unlabeled
	observation \(G_n^\circ\) whenever only its topology is relevant.
	
	\subsection{Percolation}
	
	Bernoulli bond percolation underlies both analyses.  We use \(p\) for
	structural component analysis and \(q\) for the auxiliary percolation in the
	root-finding algorithm.  For the former, independently retain each edge of
	\(G_n\) with probability \(p\in(0,1)\).  Equivalently, retain the tree and
	shortcut edges independently with probability \(p\); the retained shortcut
	layer is an independent \(\ErdosRenyi(n,p\gl/n)\) graph.
	Figure~\ref{fig:percolation-pipeline} illustrates the corresponding
	\(q\)-percolation used by the root-finding algorithm.

	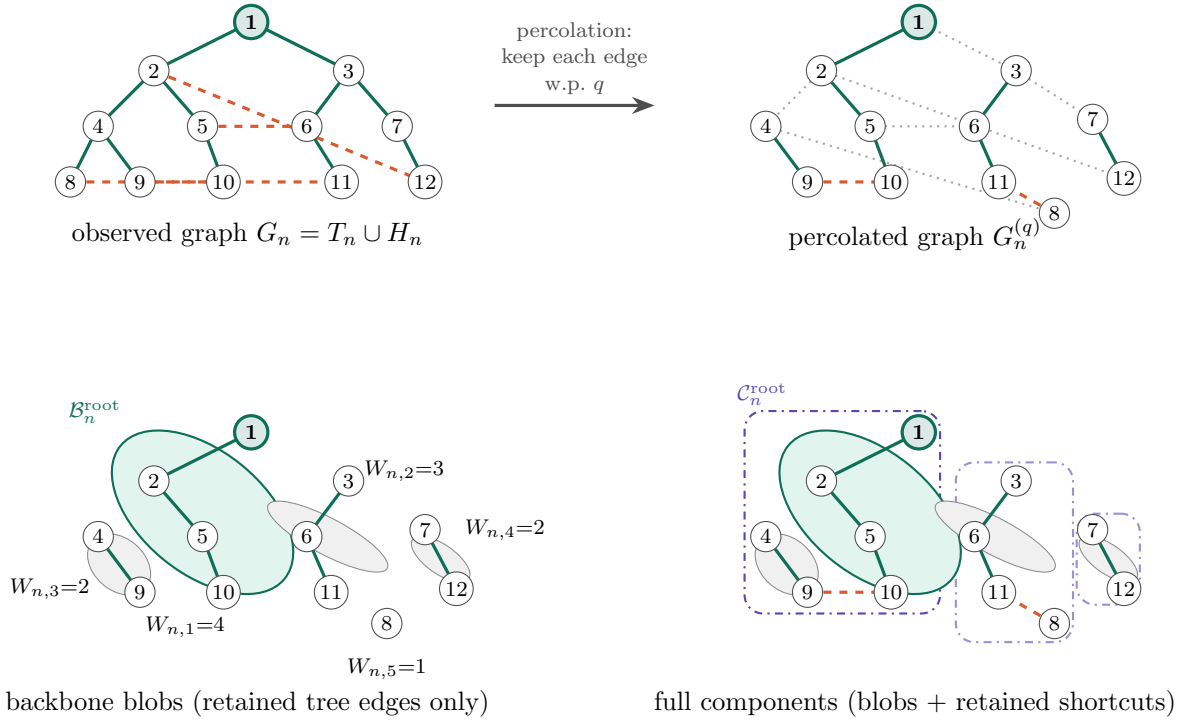
\begin{figure}[t]
		\centering
		\resizebox{\textwidth}{!}{%
			\begin{tikzpicture}[scale=0.92]
				\begin{scope}[xshift=0cm,yshift=5.9cm]
					\vertices
					\alltree{treecol,very thick}
					\allshort{shortcol,very thick,dashed}
					\drawnodes
					\node[font=\small] at (3.35,0.25) {observed graph $G_n=T_n\cup H_n$};
				\end{scope}
				\draw[-{Stealth[length=3mm]},very thick,black!70]
				(6.9,8.0) -- (9.2,8.0)
				node[midway,above,font=\scriptsize,align=center]{percolation:\\ keep each edge\\ w.p.\ $q$};
				\begin{scope}[xshift=9.6cm,yshift=5.9cm]
					\verticesB
					\deletededges{black!35,dotted,thick}
					\retainedtree{treecol,very thick}
					\retainedshort{shortcol,very thick,dashed}
					\drawnodes
					\node[font=\small] at (3.35,0.25) {percolated graph $G_n^{(q)}$};
				\end{scope}
				\begin{scope}[xshift=0cm,yshift=0cm]
					\verticesB
					\blobhulls{}
					\retainedtree{treecol,very thick}
					\drawnodes
					\node[font=\scriptsize,rootcol] at (1.15,3.6) {$\mathcal{B}_n^{\mathrm{root}}$};
					\node[font=\scriptsize] at (2.45,0.5) {$W_{n,1}{=}4$};
					\node[font=\scriptsize] at (5.6,2.75) {$W_{n,2}{=}3$};
					\node[font=\scriptsize] at (0.5,1.05) {$W_{n,3}{=}2$};
					\node[font=\scriptsize] at (7.05,1.9) {$W_{n,4}{=}2$};
					\node[font=\scriptsize] at (5.35,-0.1) {$W_{n,5}{=}1$};
					\node[font=\small] at (3.35,-0.6) {backbone blobs (retained tree edges only)};
				\end{scope}
				\begin{scope}[xshift=9.6cm,yshift=0cm]
					\begin{scope}[on background layer]
						\verticesB
						\node[draw=compcol,thick,dash dot,rounded corners=7pt,inner sep=8pt,fit=(v1)(v2)(v4)(v5)(v9)(v10)] {};
						\node[draw=compcol!60,thick,dash dot,rounded corners=7pt,inner sep=7pt,fit=(v3)(v6)(v8)(v11)] {};
						\node[draw=compcol!60,thick,dash dot,rounded corners=7pt,inner sep=6pt,fit=(v7)(v12)] {};
					\end{scope}
					\verticesB
					\blobhulls{}
					\retainedtree{treecol,very thick}
					\retainedshort{shortcol,very thick,dashed}
					\drawnodes
					\node[font=\scriptsize,compcol] at (1.15,3.85) {$\mathcal{C}_n^{\mathrm{root}}$};
					\node[font=\small] at (3.35,-0.6) {full components (blobs $+$ retained shortcuts)};
				\end{scope}
			\end{tikzpicture}%
		}
		\caption{From the observed graph to the subcritical blob picture. Top: the
			observed graph $G_n=T_n\cup H_n$ and its $q$-percolation $G_n^{(q)}$, in
			which each edge is retained independently with probability $q$ (deleted
			edges dotted). Bottom left: the backbone blobs, the connected clusters of
			retained \emph{tree} edges, ranked by mass
			$W_{n,1}\ge W_{n,2}\ge\cdots$ with ties broken by smallest vertex label;
			the blob containing vertex $1$ is the root blob
			$\mathcal{B}_n^{\mathrm{root}}$. Bottom right: the full components of
			$G_n^{(q)}$, obtained by gluing blobs along retained shortcut edges; the
			component containing the root blob is $\mathcal{C}_n^{\mathrm{root}}$.
			The displayed labels and edge types are for illustration only.}
		\label{fig:percolation-pipeline}
	\end{figure}

	Let
	\[
	\cB_{n,1},\cB_{n,2},\ldots
	\]
	be the connected components of the retained recursive-tree backbone, ranked in
	decreasing order of size, with ties broken by increasing smallest label. Put
	\[
	W_{n,i}=|\cB_{n,i}|,
	\qquad
	\cB_n(v)=\text{the retained tree blob containing }v.
	\]
	The full percolated components of \(G_n\) are obtained by adding the retained
	shortcut edges between these blobs.
	
	Conditional on the retained tree blobs and their sizes, the retained shortcut
	layer induces a random graph on the blob set: two distinct blobs \(i,j\) are
	joined when at least one retained shortcut edge connects their vertices. These
	events are independent for different blob pairs, and
	\[
	\PP(i\leftrightarrow j\mid (W_{n,a})_a)
	=
	1-\left(1-\frac{p\gl}{n}\right)^{W_{n,i}W_{n,j}}
	=
	\frac{p\gl W_{n,i}W_{n,j}}{n}
	+o\left(\frac{W_{n,i}W_{n,j}}{n}\right).
	\]
	Thus the contracted shortcut graph is a rank-one inhomogeneous random graph
	with vertex weights \(W_{n,i}\), in the sense of the general theory of
	inhomogeneous random graphs; see
	\cite{BollobasJansonRiordan2007IRG,VanDerHofstad2024RGCNII}.  Its local
	component exploration is governed, to first order, by the expected total blob
	mass reached from a size-biased blob, namely
	\[
	p\gl\,\frac1n\sum_i W_{n,i}^2.
	\]
	By Corollary~\ref{cor:backbone-susceptibility}, this quantity converges in
	probability to \(p\gl/(1-2p)\).  Define the shortcut subcriticality parameter
	\[
	\rho(p,\gl)
	:=
	\frac{p\gl}{1-2p}.
	\]
	The regime considered below is
	\[
	0<p<p_c(\gl):=\frac{1}{\gl+2},
	\qquad\text{equivalently}\qquad
	\rho(p,\gl)<1.
	\]
	In particular, \(p<1/2\).  The root-finding result uses a second, auxiliary
	percolation parameter.  After observing \(G_n^\circ\), independently retain
	each observed edge with probability
	\begin{equation}
		\label{eq:arch-q-range}
		0<q<\min\left\{\frac1{\gl+2},\frac13\right\}.
	\end{equation}
	Figure~\ref{fig:subcritical-regime-map} records the structural and
	root-finding ranges as functions of the shortcut intensity \(\gl\).
	
	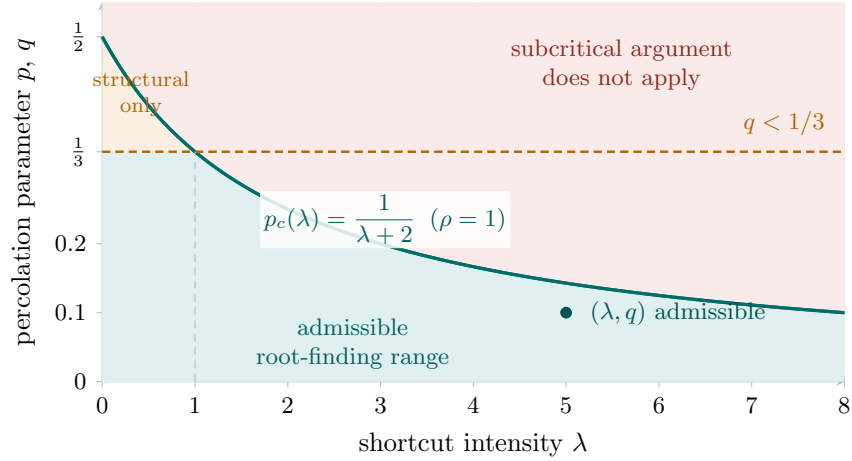
\begin{figure}[tbp]
		\centering
		\begin{tikzpicture}
			\begin{axis}[
				figaxis,
				width=11.4cm, height=6.6cm,
				axis lines=left,
				xlabel={shortcut intensity \(\gl\)},
				ylabel={percolation parameter \(p,\,q\)},
				xmin=0, xmax=8, ymin=0, ymax=0.55,
				enlargelimits=false, clip=false,
				xtick={0,1,2,3,4,5,6,7,8},
				ytick={0,0.1,0.2,0.3333,0.5},
				yticklabels={\(0\),\(0.1\),\(0.2\),\(\tfrac13\),\(\tfrac12\)},
				xmajorgrids, ymajorgrids,
				]
				\addplot[name path=curve, figteal!85!black, very thick,
				domain=0:8, samples=160] {1/(x+2)};
				\addplot[name path=qline, figorange!85!black, thick, densely dashed,
				domain=0:8, samples=2] {1/3};
				\path[name path=top]    (axis cs:0,0.55) -- (axis cs:8,0.55);
				\path[name path=bottom] (axis cs:0,0) -- (axis cs:8,0);
				
				\addplot[figred!10] fill between[of=curve and top];
				\addplot[figorange!12] fill between[of=qline and curve,
				soft clip={domain=0:1}];
				\addplot[figteal!12] fill between[of=bottom and qline,
				soft clip={domain=0:1}];
				\addplot[figteal!12] fill between[of=bottom and curve,
				soft clip={domain=1:8}];
				
				\addplot[figteal!85!black, very thick, domain=0:8, samples=160] {1/(x+2)};
				\addplot[figorange!85!black, thick, densely dashed, domain=0:8, samples=2] {1/3};
				\draw[gray!55, densely dashed] (axis cs:1,0) -- (axis cs:1,0.3333);
				
				\addplot[only marks, mark=*, mark size=2pt, figteal!65!black]
				coordinates {(5,0.10)};
				\node[anchor=west, font=\footnotesize, figteal!65!black]
				at (axis cs:5.15,0.10) {\((\gl,q)\) admissible};
				
				\node[figteal!75!black, font=\footnotesize, fill=white, fill opacity=0.85,
				text opacity=1, inner sep=1.5pt, rounded corners=1pt]
				at (axis cs:3.05,0.235)
				{\(p_c(\gl)=\dfrac{1}{\gl+2}\;\;(\rho=1)\)};
				\node[figorange!82!black, font=\footnotesize, anchor=south east]
				at (axis cs:7.9,0.3433) {\(q<1/3\)};
				\node[align=center, figteal!75!black, font=\footnotesize]
				at (axis cs:2.7,0.055) {admissible\\root-finding range};
				\node[align=center, figred!75!black, font=\footnotesize]
				at (axis cs:5.6,0.46) {subcritical argument\\does not apply};
				\node[align=center, figorange!82!black, font=\scriptsize, anchor=south]
				at (axis cs:0.42,0.37) {structural\\only};
			\end{axis}
		\end{tikzpicture}
		\caption{Structural and root-finding parameter ranges for the random recursive
			tree substrate.  The teal curve \(p_c(\gl)=1/(\gl+2)\) is the
			rank-one critical curve \(\rho(p,\gl)=1\): below it the shortcut layer is
			subcritical.  Root finding uses the smaller auxiliary range
			\(q<\min\{p_c(\gl),1/3\}\) (teal); the orange sliver
			\(1/3\le q<p_c(\gl)\) is structurally subcritical but outside the
			root-finding range used here, and above the curve (red) the subcritical
			argument does not apply.}
		\label{fig:subcritical-regime-map}
	\end{figure}
	
	Let
	\[
	\cC_{n,1}^{(q)},\cC_{n,2}^{(q)},\ldots
	\]
	be the connected components of this retained graph, ranked in decreasing
	order of size with deterministic tie-breaking.  For a graph \(\cG\), let
	\(\Comp(\cG)\) denote its set of connected components.  For a finite connected
	graph \(C\) and a vertex \(v\in C\), define the graph-Jordan score
	\[
	\psi_C(v)
	:=
	\max\{|D|:D\in\Comp(C\setminus\{v\})\}
	\]
	For \(L\ge1\), let \(\TopJ_L(C)\) consist of the
	\(\min\{L,|C|\}\) vertices of \(C\) with smallest \(\psi_C\)-values, again with
	deterministic isomorphism-invariant tie-breaking.  The root-finding output is
	\begin{equation}
		\label{eq:arch-algorithm}
		\cH_{K,L}(G_n^\circ)
		:=
		\bigcup_{i=1}^K \TopJ_L(\cC_{n,i}^{(q)}).
	\end{equation}
	This set has size at most \(KL\) and is a function of the unlabeled topology
	of the observed graph and of the auxiliary percolation randomness.
	
	\begin{rem}
		The restriction \(q<1/3\) is not expected to be sharp.  It is imposed here to
		keep the proof in a finite-second-moment regime for the decoration variables
		defined below.  Equivalently, it is the condition under which the size-biased
		Yule--Simon blob law has a finite second moment; this enters in
		Lemma~\ref{lem:external-third-susceptibility}.  The subcritical component
		comparison itself only needs \(q<1/(\gl+2)\).
	\end{rem}
	
	\section{Results}
	\label{sec:results}
	We first state the subcritical component picture and then its root-finding
	consequence.
	
	\begin{thm}[Subcritical blob picture]
		\label{thm:subcritical-blob-picture}
		Fix \(\gl>0\) and \(0<p<p_c(\gl)\). Then the following hold.
		\begin{enumeratea}
			\item For every \(\eps>0\), there is a finite \(K=K(\eps)\) such that
			\[
			\limsup_{n\to\infty}
			\PP\left(
			\rank(\cB_n(1))>K
			\right)
			<\eps .
			\]
			Here \(\rank(\cB_n(1))\) is the position of the root blob in the
			decreasing size order fixed in Section~\ref{sec:model-notation}.
			\item The normalized counts of retained recursive-tree blobs have a
			power-law tail. More precisely, if \(N_{n,k}\) denotes the number of blobs
			of size \(k\), then, for every fixed \(k\ge 1\),
			\[
			\frac{N_{n,k}}{n}
			\probc
			q_k
			:=
			\frac{1-p}{p}\,B\left(k,1+\frac1p\right),
			\]
			where \(B(\cdot,\cdot)\) is the Beta function. Thus,
			\[
			\sum_{\ell\ge k} q_\ell
			\sim
			(1-p)\Gamma\left(1+\frac1p\right)k^{-1/p}.
			\]
			The size-biased blob seen from a uniformly chosen vertex has probabilities
			\[
			\PP(W^\star=k)=kq_k,
			\]
			and, for \(p<1/2\),
			\[
			\EE W^\star=\sum_{k\ge1}k^2q_k=\frac{1}{1-2p}.
			\]
			\item For each fixed \(K\), with probability tending to one, the \(K\) largest
			full percolation components are those containing the \(K\) largest
			backbone blobs, with their order determined by the tie rule above. Moreover,
			for every fixed \(i\),
			\[
			\left|
			|\cC_{n,i}^{\full}|
			-
			\frac{W_{n,i}}{1-\rho(p,\gl)}
			\right|
			=
			o_{\PP}(n^p),
			\]
			where \(\cC_{n,i}^{\full}\) is the full percolation component containing
			\(\cB_{n,i}\).
		\end{enumeratea}
	\end{thm}
	
	\begin{rem}
		This statement includes the fixed tie rule. With high probability, no two of
		the top \(K\) blobs are connected to one another by the shortcut layer, no
		lower-rank blob has a shortcut component large enough to overtake them, and
		the component containing \(\cB_{n,i}\) has the asymptotic size displayed
		above. The deterministic rule resolves exact ties.
	\end{rem}
	
	\begin{thm}[Root finding in the unlabeled shortcut model]
		\label{thm:network-archeology-root-finding}
		Fix \(\gl>0\) and choose \(q\) satisfying \eqref{eq:arch-q-range}. For every
		\(\eps>0\), there exist finite \(K=K(\eps,q,\gl)\) and
		\(L=L(\eps,q,\gl)\) such that the algorithm
		\eqref{eq:arch-algorithm} satisfies
		\begin{equation}
			\label{eq:network-archeology-root-finding}
			\liminf_{n\to\infty}
			\PP\{1\in \cH_{K,L}(G_n^\circ)\}
			\ge 1-\eps .
		\end{equation}
		Thus the original unlabeled observation \(G_n^\circ\), before the auxiliary
		percolation, admits a root confidence set whose size is deterministically
		bounded by \(K(\eps,q,\gl)L(\eps,q,\gl)\).
	\end{thm}
	
	\section{Discussion and related work}
	\label{sec:discussion}
	
	\paragraph{\bf Network archaeology and root finding.}
	The phrase network archaeology was popularized in the algorithmic and
	computational-biology literature in \cite{NavlakhaKingsford2011NetworkArchaeology},
	where the problem was framed as recovering past states of a network from
	present-day interactions under a specified growth model.  A Bayesian
	formulation for reconstructing network histories, together with a recovery
	transition for generalized preferential attachment, was developed in
	\cite{YoungStOngeLaurenceMurphyHebertDufresneDesrosiers2019History}.  Inference
	for latent histories of shape-exchangeable growing trees was studied in
	\cite{CraneXu2020History}.  In probability, the most basic version is root
	finding: recover, up to a small confidence set, the first vertex of a growing
	random graph from its unlabeled final state.  The modern confidence-set
	formulation for uniform and preferential attachment trees was initiated in
	\cite{BubeckDevroyeLugosi2017FindingAdam}, where it was shown that the budget
	can be independent of \(n\).  This problem was placed in the broader setting
	of attachment-function-driven growing trees in
	\cite{BanerjeeBhamidi2022JordanRoot}.  That work establishes root-finding
	guarantees for Jordan centrality and persistence of the top Jordan centers.
	The present paper keeps the same confidence-set goal but changes the observed
	object from a tree to a sparse graph with cycles.
	
	\paragraph{\bf The tree-centered literature.}
	Most rigorous root-finding and seed-recovery results are for tree models.
	For random recursive and preferential attachment trees, one uses the fact that
	deleting a candidate vertex exposes descendant subtrees whose sizes can be
	compared.  This principle appears already in the Jordan-centrality algorithm
	of \cite{BubeckDevroyeLugosi2017FindingAdam}, in the persistence results of
	\cite{JogLoh2018}, and in the general attachment-function framework of
	\cite{BanerjeeBhamidi2022JordanRoot}.  Related tree problems include
	rumor-source detection and rumor centrality
	\cite{ShahZaman2011Rumors}, diffusion source confidence sets
	\cite{KhimLoh2016ConfidenceSets}, and inference of finite seeds in uniform
	attachment trees
	\cite{BubeckEldanMosselRacz2017UniformSeed,LugosiPereira2019Seed,DevroyeReddad2019Discovery}.
	Refined tail estimates used in root-finding analyses appear in
	\cite{JusticeShyamalkumar2019TailBound}.  More recent work sharpens
	confidence-set sizes, studies alternative algorithms, or estimates more of the
	latent history.  The preferential-attachment root-finding exponent was sharpened in
	\cite{ContatCurienLacroixLasalleRivoirard2023EveAdam} by using the early edge
	joining the first two vertices.  Optimal root-recovery order for uniform
	attachment and \(d\)-regular growing trees was obtained in
	\cite{AddarioBerryFontaineKhanfirLangevinTetu2024RootRecovery}, while a
	leaf-stripping algorithm gives another bounded-size root confidence set for
	uniform attachment trees in
	\cite{AddarioBerryBrandenbergerBriendBroutinLugosi2025LeafStripping}.
	Estimation of the arrival ordering in random recursive trees was studied in
	\cite{BriendGiraudLugosiSulem2024History}, and pointwise arrival-time errors
	for Jordan and iterated Jordan centralities are analyzed in
	\cite{BaumlerBriendJorritsma2026History}.  Root finding for random
	nearest-neighbor trees was analyzed in
	\cite{BrandenbergerMarcussenMosselSudan2024NearestNeighborRoot}, illustrating
	how geometric information changes the available statistics.  Five standard
	centrality measures for random recursive trees are compared in
	\cite{CollJosifovDevroyeLugosi2026Centrality}, with particular attention to
	the rank and persistence of the root.  These results give a well-developed
	theory, but their proofs are tied either to tree
	decompositions, degree martingales, or additional geometric information.
	
	\paragraph{\bf Beyond trees.}
	There are fewer rigorous network-archaeology results for non-tree graphs.  A
	model in which a latent preferential-attachment tree is observed together
	with homogeneous Erd\H{o}s--R\'enyi noise is studied in
	\cite{CraneXu2021LatentGrowth}, which develops a Bayesian inference procedure
	for early vertices and communities.  The closest result studies the same
	uniform-attachment-plus-Erd\H{o}s--R\'enyi model and proves bounded-size root
	confidence sets for \(G(n,\lambda_n/n)\) whenever
	\(\lambda_n=o(\log n)\), thereby including the fixed-\(\lambda\) regime
	considered here \cite{DevroyeLugosiMaitra2026NoisyTrees}.  Rather than introducing
	auxiliary percolation, it retains the subgraph spanned by vertices above a
	logarithmic degree threshold, shows that its largest component is a tree
	containing the root, and recovers enough exchangeability to apply Jordan
	centrality.  In particular, this method permits \(\lambda_n\to\infty\), a
	regime that at present does not appear accessible through our percolative
	coarse graining, since keeping the contracted shortcut graph subcritical would
	force the retention parameter to vanish and alter the blob scales used in our
	proof.  Thus the theorem in \cite{DevroyeLugosiMaitra2026NoisyTrees} covers a
	broader noise regime on the uniform-attachment substrate.  Root-estimation
	results for uniform random recursive DAGs and
	uniform Cooper--Frieze random networks were proved in
	\cite{BriendCalvilloLugosi2022RecursiveDAGs}, emphasizing that centrality
	methods designed for trees do not transfer directly once cycles are present.
	The distinction from the preceding methods is methodological.  Unlike posterior
	sampling \cite{CraneXu2021LatentGrowth} and high-degree filtering
	\cite{DevroyeLugosiMaitra2026NoisyTrees}, the algorithm here first applies
	independent subcritical percolation to the observed graph.  We then prove that
	the large components have a decorated-tree structure on which Jordan
	centrality is again usable.  The auxiliary percolation therefore supplies the
	coarse graining that transports tree-based confidence-set arguments to a
	cyclic graph.
	
	\paragraph{\bf Recursive-tree percolation and subcritical components.}
	The recursive-tree blob input used here is based on the Yule embedding with
	neutral mutations.  Related percolation-cluster asymptotics for random
	recursive trees are analyzed in
	\cite{Baur2014RRTPercolation,Bertoin2014RecursivePercolation}, while fixed-size
	cluster laws and largest-cluster behavior for site percolation are proved in
	\cite{GuYuan2024SitePercolationRRT}.  After contraction of the tree blobs, the
	shortcut layer is a rank-one inhomogeneous random graph.  The subcritical
	statement that large components are obtained by starting from large weights
	and attaching finite branching-process decorations parallels the theorem for
	subcritical power-law configuration models in
	\cite{Janson2008SubcriticalPowerLaw}.
	
	\paragraph{\bf General attachment substrates as future work.}
	The proof is organized so that the random-recursive-tree inputs are visible:
	the Yule--Simon blob law, the exponent \(p\), the susceptibility
	\((1-2p)^{-1}\), and the uniform-attachment Jordan estimate inside the root
	blob.  In ongoing work we consider attachment trees in which each new vertex
	chooses its parent with probability proportional to a general function \(f\)
	of the parent's degree.  The percolation-and-decoration decomposition remains
	the organizing idea, but the blob law, susceptibility, and root estimate
	require substrate-specific arguments.  Power-sublinear attachment
	\(f(k)=k^\alpha\), with \(\alpha\in(0,1)\), is a natural test class.  The
	threshold \(\alpha=1/2\) separates persistent and nonpersistent degree
	behavior: for \(\alpha>1/2\), persistent hubs can provide root information; see
	\cite{BanerjeeBhamidi2021Hubs} for the persistent-hub transition and
	\cite{BanerjeeHuang2023DegreeCentrality} for degree-centrality root confidence
	sets in growing random networks.  For \(\alpha\leq1/2\), the highest-degree
	vertices are not persistent, so the persistent-hub argument is unavailable.
	
	\paragraph{\bf Limitations and possible extensions.}
	The proof uses auxiliary randomness and requires
	\(q<\min\{1/(\gl+2),1/3\}\).  The first bound is the natural subcritical
	condition for the contracted shortcut graph, while the second is a convenient
	finite-second-moment assumption for the decoration variables.  The restriction
	\(q<1/3\) should not be interpreted as a conjectural threshold.  The proof
	derives the recursive-tree blob inputs from the Yule embedding and isolates
	the rank-one shortcut comparison as a conditional subcritical theorem, but it
	does not optimize constants or confidence-set sizes.  The auxiliary parameter
	\(q\) also introduces a natural algorithmic tradeoff: smaller \(q\) improves
	subcriticality and makes individual retained components smaller, but it
	fragments the observed graph into more components.  Understanding the optimal
	choice of \(q\), and how this choice affects the required budgets \(K\) and
	\(L\), remains open.  The next regimes are the critical window
	\(\rho(p,\gl)\approx1\), where leading blobs should interact nontrivially,
	and the supercritical regime, where the root may lie in a giant component.
	In both regimes, large components can no longer be treated as a single leading
	blob with finite decorations, so a different coarse-grained root statistic is
	required.
	
	\section{Proofs}
	\label{sec:proofs}
	
	We organize the proofs in two parts, corresponding to the two main theorems.
	Sections~\ref{sec:proof-rrt-blob-input} through
	\ref{sec:proof-subcritical-blob-picture} prove
	Theorem~\ref{thm:subcritical-blob-picture}: first the Yule embedding gives the
	recursive-tree blob asymptotics, then the shortcut layer is analyzed as a
	subcritical rank-one graph on blobs, and finally these inputs identify the
	leading full components.  Sections~\ref{sec:proof-network-archaeology} through
	\ref{sec:proof-network-archeology-root-finding} prove
	Theorem~\ref{thm:network-archeology-root-finding}: the root component is
	reduced to a decorated uniform-attachment skeleton, weighted Jordan centrality
	is controlled on that skeleton, and the resulting root estimate is transferred
	back to graph Jordan centrality in the observed percolated component.
	
	\subsection{Recursive-tree blob input}
	\label{sec:proof-rrt-blob-input}
	
	We use the standard continuous-time embedding of the random recursive tree.
	Let \(Y(t)\) be the Yule process started from one individual, in which every
	individual gives birth at rate one. The genealogical tree of \(Y(t)\), stopped
	when the population first reaches size \(n\), is distributed as \(T_n\). If
	\[
	\tau_n:=\inf\{t:Y(t)=n\},
	\]
	then
	\[
	\tau_n-\log n \weakc -\log E,
	\qquad E\sim\operatorname{Exp}(1),
	\]
	and in particular \(\tau_n=\log n+O_{\PP}(1)\).
	
	Bond percolation on the genealogy can be represented by mutations at births:
	with probability \(p\), a child keeps the type of its parent; with probability
	\(1-p\), it starts a new type. The retained tree blobs are precisely the type
	families. Thus the root blob is a Yule process with birth rate \(p\), observed
	at time \(\tau_n\), and the non-root blobs are the type families generated by
	mutations.
	
	We next prove the recursive-tree percolation input used below.  The argument
	is the standard Yule-process proof of these facts, included here to make the
	later root-finding and shortcut estimates independent of any black-box
	percolation theorem.  This Yule-mutation representation is closely related to
	the recursive-tree
	bond-percolation analyses of Bertoin~\cite{Bertoin2014RecursivePercolation}
	and Baur~\cite{Baur2014RRTPercolation}, which focus on supercritical regimes
	with \(p_n\to1\).  For fixed percolation parameter, the Yule--Simon
	cluster-size law and the \(n^p\) scale for the largest clusters are aligned
	with the fixed-\(p\) site-percolation results, and their coupling to bond
	percolation, in Gu--Yuan~\cite{GuYuan2024SitePercolationRRT}.
	
	We record three elementary consequences of the Yule embedding.  Enumerate the
	type families by their birth times, with family \(0\) denoting the root
	family.  Let \(\sigma_j\) be the birth time of family \(j\ge1\), and let
	\(Y_j^{(p)}(u)\) be a rate-\(p\) Yule process started from one individual and
	run for time \(u\).  By the branching property and the thinning of births
	into retained and mutant births, a family born at time \(\sigma_j\) evolves
	after its birth as a rate-\(p\) Yule process.  Thus, for the moment and limit
	computations below, the family sizes at time \(t\) are represented by
	\[
	Y_0^{(p)}(t),\qquad
	Y_j^{(p)}(t-\sigma_j)\ind\{\sigma_j\le t\},\quad j\ge1.
	\]
	Throughout this subsection we use the convention
	\[
	E=\lim_{t\to\infty}e^{-t}Y(t),
	\qquad E\sim\operatorname{Exp}(1).
	\]
	Thus \(e^{\tau_n}/n\to E^{-1}\) and
	\(\tau_n-\log n\to-\log E\) almost surely along the Yule construction.
	For every family born at finite time,
	\[
	e^{-pu}Y_j^{(p)}(u)\to \xi_j
	\qquad\text{a.s. as }u\to\infty,
	\]
	where the limits \((\xi_j)\) are independent exponential random variables
	with mean one.
	The early--late decomposition in this embedding is illustrated in
	Figure~\ref{fig:yule-mutation-clock-proof}.
	
	\begin{figure}[t]
		\centering
		\begin{tikzpicture}[
			x=0.85cm,y=0.75cm,
			every node/.style={font=\small},
			family/.style={circle,draw=#1!70!black,fill=#1!18,thick,minimum size=7pt},
			faint/.style={circle,draw=gray!60,fill=gray!12,thick,minimum size=6pt}
			]
			\fill[teal!7] (0,-0.55) rectangle (4.2,2.25);
			\draw[-{Latex[length=2.4mm]},thick] (0,0) -- (10.8,0);
			\node[below] at (0,0) {\(0\)};
			\node[below] at (4.2,0) {\(T\)};
			\node[below] at (9.9,0) {\(\tau_n\)};
			\draw[densely dashed] (4.2,-0.35) -- (4.2,2.15);
			\draw[densely dashed] (9.9,-0.35) -- (9.9,2.15);
			\node[teal!65!black] at (2.1,1.95) {early families};
			\node[gray!75!black] at (7.0,1.95) {late families};
			
			\draw[teal!70!black,very thick] (0,1.35) -- (9.9,1.35);
			\node[family=teal,label=left:{root}] at (0,1.35) {};
			\node[right] at (9.95,1.35) {\(E^{-p}\xi_0 n^p\)};
			
			\foreach \x/\h/\name in {1.0/0.75/\(\sigma_1\),2.0/1.05/\(\sigma_2\),3.25/0.55/\(\sigma_3\)} {
				\draw[orange!80!black,thick] (\x,0) -- (\x,\h);
				\node[family=orange,label=below:{\name}] at (\x,0) {};
				\draw[orange!80!black,thick] (\x,\h) -- (9.9,\h);
			}
			\foreach \x/\h in {5.3/0.75,6.2/1.05,7.4/0.55,8.4/0.9} {
				\draw[gray!65,thick] (\x,0) -- (\x,\h);
				\node[faint] at (\x,0) {};
				\draw[gray!65,thick,dashed] (\x,\h) -- (9.9,\h);
			}
			\node[orange!80!black,align=center,font=\scriptsize] at (2.3,-1.25)
			{finite list has explicit\\\(n^p\)-scale limits};
			\node[gray!75!black,align=center,font=\scriptsize] at (7.1,-1.25)
			{maximum becomes negligible\\after \(T\to\infty\)};
		\end{tikzpicture}
		\caption{The Yule mutation clock.  Families born before a fixed time \(T\)
			have explicit \(n^p\)-scale limits at \(\tau_n\), while the largest family
			born after \(T\) is negligible after taking \(T\to\infty\).}
		\label{fig:yule-mutation-clock-proof}
	\end{figure}
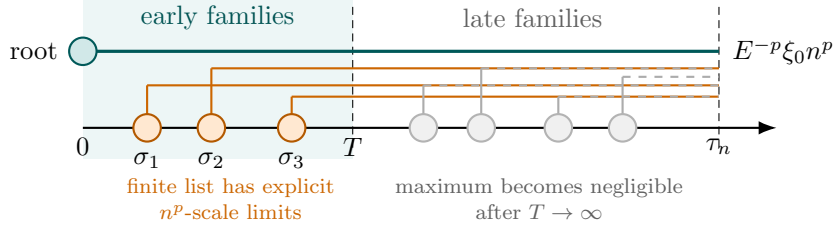
	
	\begin{lem}[Early mutation family limits]
		\label{lem:early-mutation-limits}
		For each fixed \(T<\infty\), the ranked sizes at time \(\tau_n\) of all
		families born no later than \(T\), divided by \(n^p\), converge jointly almost
		surely in the Yule construction to the ranked points of
		\[
		\left\{
		E^{-p}\xi_0
		\right\}
		\cup
		\left\{
		E^{-p}e^{-p\sigma_j}\xi_j:\sigma_j\le T
		\right\}.
		\]
		The limiting points are almost surely distinct.
	\end{lem}
	
	\begin{proof}
		There are almost surely finitely many mutation births before time \(T\).  For
		each such family, the martingale convergence above and
		\(\tau_n-\log n\to-\log E\) give
		\[
		n^{-p}Y_j^{(p)}(\tau_n-\sigma_j)
		=
		\left(\frac{e^{\tau_n}}{n}\right)^p
		e^{-p\sigma_j}
		e^{-p(\tau_n-\sigma_j)}
		Y_j^{(p)}(\tau_n-\sigma_j)
		\to
		E^{-p}e^{-p\sigma_j}\xi_j .
		\]
		For the root family the same display holds with \(\sigma_0=0\).  Since the
		limits are continuous random variables conditional on the mutation times and
		there are finitely many of them, ties occur with probability zero.
	\end{proof}
	
	\begin{lem}[Late mutation families are negligible]
		\label{lem:late-mutation-negligible}
		For every \(\eps>0\),
		\[
		\lim_{T\to\infty}\limsup_{n\to\infty}
		\PP\left(
		\max_{j:\,T<\sigma_j\le \tau_n}
		Y_j^{(p)}(\tau_n-\sigma_j)>\eps n^p
		\right)=0.
		\]
	\end{lem}
	
	\begin{proof}
		Choose \(r>1/p\).  A rate-\(p\) Yule process started from one
		particle has geometric law
		\(\PP(Y^{(p)}(u)=\ell)=e^{-pu}(1-e^{-pu})^{\ell-1}\), \(\ell\ge1\);
		hence, for this fixed \(r\),
		\[
		\EE [Y^{(p)}(u)]^r\le C_r e^{pru},\qquad u\ge0.
		\]
		Fix \(A>0\), put \(t_n=\log n+A\), and write
		\[
		M_{T,n}:=\max_{j:\,T<\sigma_j\le \tau_n}
		Y_j^{(p)}(\tau_n-\sigma_j).
		\]
		Then \(\PP(M_{T,n}>\eps n^p)\le \PP(\tau_n>t_n)+
		\PP(M_{T,n}>\eps n^p,\tau_n\le t_n)\).  On the event
		\(\{\tau_n\le t_n\}\), every family born before \(\tau_n\) has only grown
		further by time \(t_n\), so
		\[
		M_{T,n}\le
		\max_{j:\,T<\sigma_j\le t_n}Y_j^{(p)}(t_n-\sigma_j).
		\]
		Consequently, by the union bound, Markov's inequality, and the compensation
		formula, using that mutation births occur at predictable rate
		\((1-p)Y(s)\,ds\),
		\[
		\begin{aligned}
			\PP(M_{T,n}>\eps n^p,\tau_n\le t_n)
			&\le C_{\eps,r}n^{-pr}\int_T^{t_n}
			\EE Y(s)\,e^{pr(t_n-s)}\,ds\\
			&=
			C_{\eps,r}n^{-pr}
			\int_T^{t_n}e^s e^{pr(t_n-s)}\,ds .
		\end{aligned}
		\]
		Since \(t_n=\log n+A\) and \(pr>1\),
		\[
		C_{\eps,r}n^{-pr}
		\int_T^{t_n}e^s e^{pr(t_n-s)}\,ds
		\le C_{\eps,r,A}e^{-(pr-1)T}.
		\]
		Thus, for every fixed \(A\),
		\[
		\limsup_{n\to\infty}\PP(M_{T,n}>\eps n^p)\le
		\limsup_{n\to\infty}\PP(\tau_n>\log n+A)
		+C_{\eps,r,A}e^{-(pr-1)T}.
		\]
		Since \(\tau_n-\log n\to-\log E\) almost surely, the first term is at most
		\(\PP(-\log E>A)\).  Letting \(T\to\infty\) and then \(A\to\infty\)
		proves the claim.
	\end{proof}
	
	\begin{lem}[Fixed-size family counts and uniform family moments]
		\label{lem:fixed-size-family-counts}
		Recall that \(W_{n,i}=|\cB_{n,i}|\) are the ranked retained-tree blob
		sizes.  The following estimates hold.
		\begin{enumeratea}
			\item Let \(N_k(t)\) be the number of type families of size \(k\) at Yule
			time \(t\).  Then, for each fixed \(k\ge1\),
			\[
			\frac{N_k(t)}{Y(t)}
			\to
			q_k
			:=
			\frac{1-p}{p}B\left(k,1+\frac1p\right)
			\qquad\text{a.s.},
			\]
			where \(B(\cdot,\cdot)\) is the Beta function.  If
			\(N_{n,k}:=\#\{i:W_{n,i}=k\}\), then \(N_{n,k}/n\to q_k\) almost surely.
			\item For every fixed \(r>0\),
			\begin{equation}
				\label{eq:uniform-family-moments}
				\sum_i W_{n,i}^r
				=
				\begin{cases}
					O_{\PP}(n), & pr<1,\\
					O_{\PP}(n\log n), & pr=1,\\
					O_{\PP}(n^{pr}), & pr>1.
				\end{cases}
			\end{equation}
			In particular, for every \(\eta>0\),
			\[
			\sum_i W_{n,i}^r
			=
			O_{\PP}\left(n^{\max\{1,rp\}+\eta}\right).
			\]
			\item For every \(\eta\in(0,1/p)\),
			\begin{equation}
				\label{eq:uniform-blob-count-tail}
				\sup_{x\ge1}
				\frac{x^{1/p-\eta}}{n}
				\#\{i:W_{n,i}\ge x\}
				=O_{\PP}(1).
			\end{equation}
			\item If \(pr>1\), then the ranked \(r\)-mass is uniformly tight at scale
			\(n^{pr}\): for every \(\eps>0\),
			\begin{equation}
				\label{eq:ranked-family-r-tail}
				\lim_{L\to\infty}\limsup_{n\to\infty}
				\PP\left(
				n^{-pr}\sum_{i>L}W_{n,i}^r>\eps
				\right)=0.
			\end{equation}
		\end{enumeratea}
	\end{lem}
	
	\begin{proof}
		We prove the four assertions in order.  A size-\(k\) family grows to size
		\(k+1\) at rate \(pk\), and new size-one families are born at rate
		\((1-p)Y(t)\).  Thus, for fixed \(k\), there are counting-process
		martingales \(M_k\) such that
		\[
		\begin{aligned}
			N_1(t)
			&=1+\int_0^t\{(1-p)Y(s)-pN_1(s)\}\,ds+M_1(t),\\
			N_k(t)
			&=\int_0^t\{p(k-1)N_{k-1}(s)-pkN_k(s)\}\,ds+M_k(t),
			\qquad k\ge2.
		\end{aligned}
		\]
		The predictable quadratic variations satisfy
		\(\EE d\langle M_k\rangle_s\le C_k\EE Y(s)\,ds=C_ke^s\,ds\), since
		\(N_j(s)\le Y(s)\).  Solving these linear martingale equations by
		integrating factors, set
		\[
		R_1(t):=e^{-pt}\int_0^t e^{ps}\,dM_1(s),
		\qquad
		R_k(t):=e^{-pkt}\int_0^t e^{pks}\,dM_k(s),\quad k\ge2 .
		\]
		For \(a_k=pk\), put \(H_k(t)=\int_0^t e^{a_ks}\,dM_k(s)\), so that
		\(R_k(t)=e^{-a_kt}H_k(t)\).  Doob's inequality on unit intervals and the
		quadratic-variation bound give, for every \(\eps>0\),
		\[
		\begin{aligned}
			\sum_{m\ge1}
			\PP\left(
			\sup_{m\le t\le m+1}e^{-t}|R_k(t)|>\eps
			\right)
			&\le
			\eps^{-2}\sum_{m\ge1}
			e^{-2(1+a_k)m}
			\EE\sup_{0\le u\le m+1}|H_k(u)|^2  \\
			&\le
			C_{k,\eps}\sum_{m\ge1}
			e^{-2(1+a_k)m}
			\int_0^{m+1} e^{(2a_k+1)s}\,ds
			<\infty .
		\end{aligned}
		\]
		By Borel--Cantelli,
		\[
		e^{-t}R_k(t)\to0\qquad\text{a.s.}
		\]
		The integrated identities are
		\[
		\begin{aligned}
			N_1(t)
			&=e^{-pt}
			+(1-p)\int_0^t e^{-p(t-s)}Y(s)\,ds+R_1(t),\\
			N_k(t)
			&=p(k-1)\int_0^t e^{-pk(t-s)}N_{k-1}(s)\,ds+R_k(t),
			\qquad k\ge2.
		\end{aligned}
		\]
		Because \(e^{-t}Y(t)\to E>0\) almost surely, the first formula gives
		\[
		e^{-t}N_1(t)\to
		(1-p)E\int_0^\infty e^{-(1+p)u}\,du
		=
		\frac{1-p}{1+p}E
		\qquad\text{a.s.}
		\]
		Indeed, after setting \(u=t-s\), the integrand is
		\[
		e^{-pu}e^{-t}Y(t-u)
		=
		e^{-(1+p)u}e^{-(t-u)}Y(t-u),
		\]
		and pathwise dominated convergence applies on compact \(u\)-intervals; the
		tail \(u>A\) is bounded by \(C(\omega)e^{-pA}\) for all large \(t\).
		
		Inductively, if \(e^{-t}N_{k-1}(t)\to q_{k-1}E\) almost surely, then, after
		setting \(u=t-s\), the main term in \(e^{-t}N_k(t)\) is
		\[
		p(k-1)\int_0^t e^{-pku}
		e^{-t}N_{k-1}(t-u)\,du
		=
		p(k-1)\int_0^t e^{-(1+pk)u}
		e^{-(t-u)}N_{k-1}(t-u)\,du .
		\]
		Using \(N_{k-1}(v)\le Y(v)\), the same dominated-convergence argument gives
		\(e^{-t}N_k(t)\to q_kE\) almost surely, where
		\[
		q_k
		=
		p(k-1)q_{k-1}\int_0^\infty e^{-(1+pk)u}\,du
		=
		\frac{p(k-1)}{1+pk}\,q_{k-1}.
		\]
		Hence
		\[
		\frac{N_k(t)}{Y(t)}
		\to q_k
		\qquad\text{a.s.}
		\]
		The initial value is
		\[
		q_1=\frac{1-p}{1+p},
		\qquad
		q_k=\frac{p(k-1)}{1+pk}q_{k-1}.
		\]
		Since
		\[
		\frac{B(k,1+1/p)}{B(k-1,1+1/p)}
		=
		\frac{k-1}{k+1/p},
		\]
		this recurrence is solved by
		\[
		q_k=\frac{1-p}{p}B\left(k,1+\frac1p\right).
		\]
		The same convergence holds at the random times \(\tau_n\), since
		\(\tau_n\to\infty\) a.s. and \(Y(\tau_n)=n\).  As \(N_{n,k}=N_k(\tau_n)\),
		this proves part (a).
		
		For part (b), let \(Y^{(p)}\) be a rate-\(p\) Yule process started from one
		particle and put
		\[
		m_r^{(p)}(u):=\EE\big[Y^{(p)}(u)^r\big].
		\]
		The law of \(Y^{(p)}(u)\) is geometric:
		\[
		\PP\big(Y^{(p)}(u)=\ell\big)
		=a(1-a)^{\ell-1},
		\qquad a=e^{-pu}.
		\]
		For \(0<r\le1\), Jensen's inequality gives
		\(m_r^{(p)}(u)\le e^{pu}\).  For \(r>1\), comparison of the geometric
		sum with an exponential integral gives
		\[
		a\sum_{\ell\ge1}\ell^r(1-a)^{\ell-1}
		\le C_r a^{-r}.
		\]
		Consequently, for every fixed \(r>0\),
		\begin{equation}
			\label{eq:yule-r-moment}
			m_r^{(p)}(u)
			\le C_r e^{p(r\vee1)u},
			\qquad u\ge0.
		\end{equation}
		
		Let \(S_r(t)\) be the sum of the \(r\)th powers of all family sizes at time
		\(t\), including the ancestral family.  The ancestral family contributes
		\(m_r^{(p)}(t)\) in expectation.  Mutation families are born with predictable
		rate \((1-p)Y(s)\,ds\), and a family born at time \(s\) then evolves as an
		independent rate-\(p\) Yule process for time \(t-s\).  Therefore
		\begin{equation}
			\label{eq:family-moment-compensator}
			\begin{aligned}
				\EE S_r(t)
				&=m_r^{(p)}(t)
				+(1-p)\int_0^t \EE Y(s)\,m_r^{(p)}(t-s)\,ds \\
				&=m_r^{(p)}(t)
				+(1-p)\int_0^t e^s m_r^{(p)}(t-s)\,ds .
			\end{aligned}
		\end{equation}
		Using \eqref{eq:yule-r-moment} in
		\eqref{eq:family-moment-compensator}, the integral term is bounded by
		\(C_r e^{p(r\vee1)t}\int_0^t e^{(1-p(r\vee1))s}\,ds\); evaluating this
		elementary integral gives the three cases below, with \(0<r\le1\) included
		in the first case since \(p<1\).  Thus
		\begin{equation}
			\label{eq:deterministic-time-family-moments}
			\EE S_r(t)
			\le C_r
			\begin{cases}
				e^t, & pr<1,\\
				(1+t)e^t, & pr=1,\\
				e^{prt}, & pr>1.
			\end{cases}
		\end{equation}
		
		Fix \(A>0\) and set \(t_n=\log n+A\).  On
		\(\{\tau_n\le t_n\}\), monotonicity of every family size and of the number
		of families gives
		\[
		\sum_iW_{n,i}^r=S_r(\tau_n)\le S_r(t_n).
		\]
		Since \(\tau_n-\log n\to-\log E\) almost surely, the probability of
		\(\{\tau_n>t_n\}\) can be made arbitrarily small by first taking \(n\)
		large and then \(A\) large.  Markov's inequality applied to
		\eqref{eq:deterministic-time-family-moments} at the corresponding scale
		\(n\), \(n\log n\), or \(n^{pr}\) proves
		\eqref{eq:uniform-family-moments}.
		
		Part (c) follows from part (b) by Markov's inequality in size space.  Fix
		\(\eta\in(0,1/p)\) and set \(a=1/p-\eta\).  Since \(pa<1\), the first line
		of \eqref{eq:uniform-family-moments} gives
		\(\sum_iW_{n,i}^a=O_{\PP}(n)\), while deterministically
		\[
		\#\{i:W_{n,i}\ge x\}
		\le x^{-a}\sum_iW_{n,i}^a,
		\qquad x\ge1.
		\]
		Taking the supremum over \(x\) proves the claimed uniform tail.
		
		It remains to prove part (d).  Assume \(pr>1\), and let
		\[
		S_{r,>T}(t)
		:=
		\sum_{j:\,T<\sigma_j\le t}
		\big(Y_j^{(p)}(t-\sigma_j)\big)^r .
		\]
		The same compensation argument, now restricted to mutation times after
		\(T\), gives, for \(t_n=\log n+A\),
		\[
		\begin{aligned}
			\EE S_{r,>T}(t_n)
			&\le
			C_r\int_T^{t_n}e^s e^{pr(t_n-s)}\,ds \\
			&\le
			C_{r,A}n^{pr}e^{-(pr-1)T}.
		\end{aligned}
		\]
		As above, on \(\{\tau_n\le t_n\}\) the late-family mass at \(\tau_n\) is
		bounded by \(S_{r,>T}(t_n)\).  Hence, for each fixed \(A\),
		\[
		\begin{aligned}
			&\limsup_{n\to\infty}
			\PP\left(
			n^{-pr}
			\sum_{j:\,T<\sigma_j\le\tau_n}
			\big(Y_j^{(p)}(\tau_n-\sigma_j)\big)^r
			>\eps
			\right)\\
			&\qquad\le
			\limsup_{n\to\infty}\PP(\tau_n>\log n+A)
			+C_{\eps,r,A}e^{-(pr-1)T}.
		\end{aligned}
		\]
		Letting \(T\to\infty\) and then \(A\to\infty\) gives
		\begin{equation}
			\label{eq:late-family-r-mass}
			\lim_{T\to\infty}\limsup_{n\to\infty}
			\PP\left(
			n^{-pr}
			\sum_{j:\,T<\sigma_j\le\tau_n}
			\big(Y_j^{(p)}(\tau_n-\sigma_j)\big)^r
			>\eps
			\right)=0.
		\end{equation}
		Let \(K_T\) be the number of families, including the ancestral family, born
		by time \(T\).  Then \(K_T<\infty\) almost surely.  On \(\{K_T\le L\}\),
		the sum of the \(r\)th powers outside the largest \(L\) families is no
		larger than the sum outside any prescribed collection of at most \(L\)
		families; choosing all families born by time \(T\) gives
		\[
		\sum_{i>L}W_{n,i}^r
		\le
		\sum_{j:\,T<\sigma_j\le\tau_n}
		\big(Y_j^{(p)}(\tau_n-\sigma_j)\big)^r.
		\]
		Indeed, the largest \(L\) families maximize the retained \(r\)-mass among all
		collections of \(L\) families, so removing them leaves no more \(r\)-mass than
		removing the \(K_T\) early families.  Therefore, for every \(\eps>0\),
		\[
		\begin{aligned}
			\PP\left(
			n^{-pr}\sum_{i>L}W_{n,i}^r>\eps
			\right)
			&\le
			\PP(K_T>L) \\
			&\quad+
			\PP\left(
			n^{-pr}
			\sum_{j:\,T<\sigma_j\le\tau_n}
			\big(Y_j^{(p)}(\tau_n-\sigma_j)\big)^r>\eps
			\right).
		\end{aligned}
		\]
		For fixed \(T\), the first term vanishes as \(L\to\infty\).  Then
		\eqref{eq:late-family-r-mass} sends the second term to zero as
		\(T\to\infty\).  This proves
		\eqref{eq:ranked-family-r-tail}.
	\end{proof}
	
	\begin{prop}[RRT percolation blob asymptotics]
		\label{prop:rrt-blob-input}
		Fix \(p\in(0,1)\).
		\begin{enumeratea}
			\item The root blob satisfies
			\[
			n^{-p}|\cB_n(1)|\to Z_1
			\qquad\text{a.s.},
			\]
			where \(Z_1\) is non-degenerate and strictly positive.
			\item For every fixed \(m\), the vector of the \(m\) largest ranked blob sizes
			satisfies
			\[
			n^{-p}(W_{n,1},\ldots,W_{n,m})
			\to
			(W_1,\ldots,W_m)
			\qquad\text{a.s.},
			\]
			where \(W_1>W_2>\cdots>W_m>0\) almost surely after the same tie convention.
			Moreover, the ranked sequence is tight in the product topology and
			\[
			\lim_{K\to\infty}\limsup_{n\to\infty}
			\PP\left(
			|\cB_n(1)|<W_{n,K}
			\right)=0.
			\]
			\item If \(N_{n,k}\) is the number of blobs of size \(k\), then for every
			fixed \(k\ge1\),
			\[
			\frac{N_{n,k}}{n}
			\to
			\frac{1-p}{p}B\left(k,1+\frac1p\right)
			\qquad\text{a.s.}
			\]
			\item If \(U_n\) is uniform on \([n]\), independent of \(T_n\) and the
			percolation, then
			\[
			|\cB_n(U_n)|\weakc W^\star,
			\qquad
			\PP(W^\star=k)
			=
			k\frac{1-p}{p}B\left(k,1+\frac1p\right).
			\]
		\end{enumeratea}
	\end{prop}
	
	\begin{proof}
		The root assertion is the root-family almost-sure limit in the Yule
		embedding; explicitly \(Z_1=E^{-p}\xi_0\).
		
		For the ranked-vector statement, first fix \(T\).
		Lemma~\ref{lem:early-mutation-limits} gives almost-sure joint convergence of
		the ranked early families.  We now upgrade the late-family truncation to an
		almost-sure one.  Let
		\[
		X_j:=\sup_{u\ge0}e^{-pu}Y_j^{(p)}(u).
		\]
		For any \(r>1/p\), Doob's \(L^r\) inequality for the nonnegative family
		martingales gives \(\EE X_j^r\le C_r\).  The compensation formula for mutation
		births gives
		\[
		\EE\sum_{j\ge1}e^{-pr\sigma_j}X_j^r
		\le
		C_r(1-p)\int_0^\infty e^{-prs}\EE Y(s)\,ds
		<\infty .
		\]
		Hence \(\sum_{j\ge1}e^{-pr\sigma_j}X_j^r<\infty\) almost surely, and therefore
		\[
		\sup_{j:\sigma_j>T}e^{-p\sigma_j}X_j\to0
		\qquad\text{a.s.}
		\]
		For \(t\ge\sigma_j\),
		\[
		e^{-pt}Y_j^{(p)}(t-\sigma_j)
		\le
		e^{-p\sigma_j}X_j.
		\]
		Together with \(e^{\tau_n}/n\to E^{-1}\), this shows that families born after
		\(T\) contribute no \(n^p\)-scale mass as \(T\to\infty\), almost surely.
		
		The finite early-family convergence, this almost-sure late-tail bound, and
		the almost-sure absence of ties imply that every fixed initial ranked segment
		converges almost surely to the ranked points of
		\[
		\left\{E^{-p}\xi_0\right\}
		\cup
		\left\{E^{-p}e^{-p\sigma_j}\xi_j:j\ge1\right\}.
		\]
		The limiting points are locally finite and almost surely distinct by the
		same conditional-continuity argument as in
		Lemma~\ref{lem:early-mutation-limits}.
		Tightness in the product topology follows from this convergence.
		
		The root-rank tightness follows from the same representation.  Conditional on
		the root limit \(E^{-p}\xi_0>0\), only finitely many points of the limiting
		sequence exceed \(E^{-p}\xi_0/2\).  The convergence of ranked early families and
		the late-family truncation imply that the number of blobs larger than the root
		blob is tight.  Since the root blob has smallest label \(1\), the deterministic
		tie rule ranks it before every other blob of the same size; hence controlling
		the number of strictly larger blobs controls its rank.
		
		Part (c) is Lemma~\ref{lem:fixed-size-family-counts} evaluated at
		\(\tau_n\).  Finally, conditional on the blob partition,
		\[
		\PP(|\cB_n(U_n)|=k)
		=
		\frac{kN_{n,k}}{n}.
		\]
		Using the fixed-size convergence and the identity
		\(\sum_{k\ge1}kq_k=1\), the law of \(|\cB_n(U_n)|\) converges to the displayed
		size-biased distribution.
	\end{proof}
	
	\begin{cor}[Backbone susceptibility]
		\label{cor:backbone-susceptibility}
		For \(p<1/2\),
		\[
		\frac1n\sum_i W_{n,i}^2
		=
		\EE\left[|\cB_n(U_n)|\mid T_n,\text{ percolation}\right]
		\probc
		\frac{1}{1-2p}.
		\]
	\end{cor}
	
	\begin{proof}
		Conditional on the blob partition, a uniform vertex lies in blob \(i\) with
		probability \(W_{n,i}/n\), so
		\[
		\EE\left[|\cB_n(U_n)|\mid T_n,\text{ percolation}\right]
		=\frac1n\sum_iW_{n,i}^2.
		\]
		It remains to identify the limit of the last display.
		
		Write \(N_{n,k}=\#\{i:W_{n,i}=k\}\) and
		\[
		q_k=\frac{1-p}{p}B\left(k,1+\frac1p\right).
		\]
		For every fixed \(K\), Lemma~\ref{lem:fixed-size-family-counts}(a) gives
		\[
		\frac1n\sum_iW_{n,i}^2\ind\{W_{n,i}\le K\}
		=
		\sum_{k=1}^K k^2\frac{N_{n,k}}n
		\to
		\sum_{k=1}^K k^2q_k
		\qquad\text{a.s.}
		\]
		We now show that the truncation error is negligible uniformly as
		\(K\to\infty\).  Choose \(\eta>0\) so small that
		\[
		a:=\frac1p-\eta>2 .
		\]
		By Lemma~\ref{lem:fixed-size-family-counts}(c),
		\[
		Z_n:=\sup_{x\ge1}
		\frac{x^a}{n}\#\{i:W_{n,i}\ge x\}
		=O_{\PP}(1).
		\]
		For \(K\ge1\), a dyadic decomposition gives
		\[
		\begin{aligned}
			\frac1n\sum_iW_{n,i}^2\ind\{W_{n,i}>K\}
			&\le
			\sum_{\ell\ge0}
			\frac{(2^{\ell+1}K)^2}{n}
			\#\{i:W_{n,i}>2^\ell K\} \\
			&\le
			4Z_nK^{2-a}\sum_{\ell\ge0}2^{(2-a)\ell}.
		\end{aligned}
		\]
		Since \(a>2\), the geometric sum is finite and the right hand side tends to
		zero in probability as \(K\to\infty\).  The computation below shows that
		\(\sum_{k\ge1}k^2q_k<\infty\).  Combining this tail bound with the fixed-\(K\)
		convergence gives
		\[
		\frac1n\sum_iW_{n,i}^2
		\probc
		\sum_{k\ge1}k^2q_k .
		\]
		Finally, by Tonelli's theorem, the beta integral, and
		\(\sum_{k\ge1}k^2x^{k-1}=(1+x)/(1-x)^3\),
		\[
		\begin{aligned}
			\sum_{k\ge1}k^2q_k
			&=
			\frac{1-p}{p}
			\int_0^1
			\frac{1+x}{(1-x)^3}(1-x)^{1/p}\,dx \\
			&=
			\frac{1-p}{p}
			\left\{
			\frac{2}{1/p-2}-\frac{1}{1/p-1}
			\right\}
			=
			\frac{1}{1-2p}.
		\end{aligned}
		\]
		This proves the claimed convergence.
	\end{proof}
	
	\subsection{Shortcut graph on blobs}
	\label{sec:proof-shortcut-graph}
	
	Condition on the retained tree blobs and their sizes
	\[
	\vw_n=(W_{n,1},W_{n,2},\ldots).
	\]
	For two distinct blobs \(i\ne j\), the retained shortcut layer connects them
	with probability
	\[
	\hat p_{ij}
	=
	1-\left(1-\frac{p\gl}{n}+O(n^{-2})\right)^{W_{n,i}W_{n,j}}
	=
	1-\exp\left\{
	-\frac{p\gl W_{n,i}W_{n,j}}{n}
	+o\left(\frac{W_{n,i}W_{n,j}}n\right)
	\right\}.
	\]
	The events are independent over unordered pairs of blobs, because they depend
	on disjoint sets of Erd\H{o}s--R\'enyi edges. Loops inside a blob do not change
	components.
	
	Thus, conditional on \(\vw_n\), the shortcut graph is a rank-one
	inhomogeneous random graph on blob vertices with weights \(W_{n,i}\) and
	kernel \(p\gl xy/n\). The mean number of new blob vertices reached by crossing
	a shortcut from a size-biased blob is
	\[
	p\gl\frac{\sum_i W_{n,i}^2}{n}
	\probc
	\frac{p\gl}{1-2p}
	=
	\rho(p,\gl).
	\]
	The breadth-first shortcut exploration behind the next lemma is summarized in
	Figure~\ref{fig:rank-one-exploration-proof}.
	
	\begin{figure}[t]
		\centering
		\begin{tikzpicture}[
			x=0.95cm,y=0.85cm,
			every node/.style={font=\small},
			seed/.style={circle,draw=teal!70!black,fill=teal!18,thick,minimum size=25pt},
			blob/.style={circle,draw=orange!75!black,fill=orange!15,thick,minimum size=#1},
			edge/.style={gray!70,densely dashed,thick}
			]
			\node[seed,label=below:{\(W\)}] (s) at (0,0) {};
			\node at (0,1.85) {generation \(0\)};
			
			\node[blob=19pt] (a1) at (2.5,0.75) {};
			\node[blob=15pt] (a2) at (2.55,-0.05) {};
			\node[blob=12pt] (a3) at (2.35,-0.85) {};
			\node at (2.5,1.85) {generation \(1\)};
			\node[below] at (2.5,-1.25) {\(\rho W+o(n^p)\)};
			
			\node[blob=14pt] (b1) at (5.0,0.55) {};
			\node[blob=11pt] (b2) at (5.1,-0.25) {};
			\node[blob=9pt] (b3) at (4.8,-0.8) {};
			\node at (5.0,1.85) {generation \(2\)};
			\node[below] at (5.0,-1.25) {\(\rho^2 W+o(n^p)\)};
			
			\node[blob=10pt] (c1) at (7.25,0.35) {};
			\node[blob=8pt] (c2) at (7.35,-0.45) {};
			\node at (7.25,1.85) {tail};
			\node[below] at (7.25,-1.25) {\(\rho^r W\)};
			
			\foreach \u in {a1,a2,a3} \draw[edge] (s) -- (\u);
			\foreach \u/\v in {a1/b1,a1/b2,a2/b2,a3/b3} \draw[edge] (\u) -- (\v);
			\foreach \u/\v in {b1/c1,b2/c1,b3/c2} \draw[edge] (\u) -- (\v);
			
			\draw[-{Latex[length=2.4mm]},thick] (0.65,1.05) -- node[above] {\(\rho<1\)} (1.85,1.05);
			\draw[-{Latex[length=2.4mm]},thick] (3.15,1.05) -- node[above] {\(\rho\)} (4.35,1.05);
			\draw[-{Latex[length=2.4mm]},thick] (5.65,1.05) -- node[above] {\(\rho\)} (6.75,1.05);
		\end{tikzpicture}
		\caption{Shortcut exploration from one large blob.  In the subcritical regime,
			successive shortcut generations have masses approximately
			\(W,\rho W,\rho^2W,\ldots\), giving the amplification factor
			\((1-\rho)^{-1}\).}
		\label{fig:rank-one-exploration-proof}
	\end{figure}
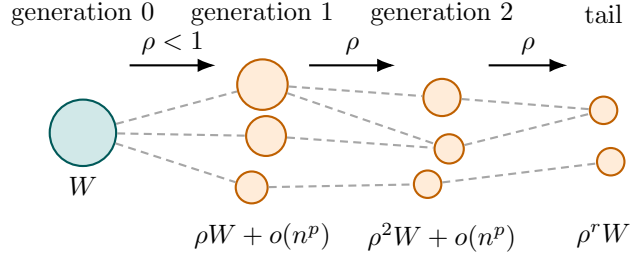
	
	\begin{lem}[Subcritical exploration from one large blob]
		\label{lem:subcritical-exploration}
		Assume \(p<p_c(\gl)\). Let \(i=i(n)\) be any blob index with
		\[
		W_{n,i}=O_{\PP}(n^p).
		\]
		Let \(F_{n,0}^{(i)}=\{i\}\), and for \(r\ge1\) let
		\(F_{n,r}^{(i)}\) be the set of blob indices first reached at shortcut
		distance \(r\) from \(\cB_{n,i}\) in the contracted shortcut graph.  Put
		\[
		A_{n,r}^{(i)}
		:=
		\sum_{j\in F_{n,r}^{(i)}}W_{n,j},
		\qquad r\ge0,
		\]
		so \(A_{n,0}^{(i)}=W_{n,i}\).  Then
		\[
		|\cC_{n,i}^{\full}|
		=
		\sum_{r\ge0} A_{n,r}^{(i)}
		=
		W_{n,i}+\sum_{r\ge1}A_{n,r}^{(i)},
		\]
		and
		\[
		A_{n,r}^{(i)}
		=
		\rho(p,\gl)^r W_{n,i}
		+
		o_{\PP}(n^p)
		\]
		for each fixed \(r\). Consequently
		\[
		|\cC_{n,i}^{\full}|
		=
		\frac{W_{n,i}}{1-\rho(p,\gl)}
		+
		o_{\PP}(n^p).
		\]
	\end{lem}
	
	\begin{proof}
		Write \(\beta=p\gl\) and \(\rho=\rho(p,\gl)\).  We condition on the retained
		tree blobs and their sizes.  All shortcut probabilities and expectations
		below are conditional on this weight sequence, and at the end we remove the
		conditioning using the high-probability estimates from the preceding
		subsection.
		
		We first record the one-step estimate used for fixed generations.  Let
		\(\mathcal H\) be the information revealed by a breadth-first exploration up
		to some generation: the explored set \(D\) of blob indices, the current
		frontier \(F\subset D\), and all shortcut edges queried so far.  Put
		\[
		a:=\sum_{u\in F}W_{n,u},
		\qquad
		d:=\sum_{u\in D}W_{n,u}.
		\]
		Suppose \(a=O(n^p)\) and \(d=O(n^p)\).  For an undiscovered blob
		\(j\notin D\), let
		\[
		q_j
		:=
		\PP(j\text{ is reached from }F\mid\mathcal H)
		=
		1-\prod_{u\in F}(1-\hat p_{uj}),
		\]
		where \(\hat p_{uj}\) is the shortcut connection probability between blobs
		\(u\) and \(j\).  On events whose probabilities tend to one,
		\[
		\max_jW_{n,j}=O(n^p),\qquad
		\frac1n\sum_jW_{n,j}^2\to\frac1{1-2p},
		\qquad
		\sum_jW_{n,j}^3
		=
		O\left(n^{\max\{1,3p\}+\eta}\right),
		\]
		where \(\eta>0\) will be chosen small.  Since \(p<1/2\),
		\[
		\frac{a}{n}\max_jW_{n,j}=o(1)
		\]
		on these events.  Uniformly for \(j\notin D\),
		\[
		\begin{aligned}
			q_j
			&=
			1-\prod_{u\in F}(1-\hat p_{uj})  \\
			&=
			\frac{\beta aW_{n,j}}{n}
			+
			O\left(\frac{a^2W_{n,j}^2}{n^2}\right)
			+
			o\left(\frac{aW_{n,j}}{n}\right).
		\end{aligned}
		\]
		Indeed,
		\(\hat p_{uj}=\beta W_{n,u}W_{n,j}/n+
		o(W_{n,u}W_{n,j}/n)\) uniformly because
		\(\max_{u,j}W_{n,u}W_{n,j}/n=o(1)\), and expanding
		\(1-\prod_{u\in F}(1-\hat p_{uj})\) gives the displayed second-order term.
		
		Let \(A_{\rm new}\) be the mass of the next frontier.  Conditional on
		\(\mathcal H\),
		\[
		A_{\rm new}
		=
		\sum_{j\notin D}W_{n,j}\xi_j,
		\qquad
		\PP(\xi_j=1\mid\mathcal H)=q_j,
		\]
		where the indicators \((\xi_j)_{j\notin D}\) are conditionally independent.
		Therefore
		\[
		\begin{aligned}
			\EE[A_{\rm new}\mid\mathcal H]
			&=
			\frac{\beta a}{n}\sum_{j\notin D}W_{n,j}^2
			+
			O\left(
			\frac{a^2}{n^2}\sum_jW_{n,j}^3
			\right)
			+
			o(a).
		\end{aligned}
		\]
		Removing the already discovered blobs changes the second-moment sum by only
		\[
		\sum_{j\in D}W_{n,j}^2
		\le
		\left(\max_jW_{n,j}\right)d
		=
		O(n^{2p})
		=
		o(n).
		\]
		The second-order term is
		\[
		O\left(
		n^{2p-2+\max\{1,3p\}+\eta}
		\right)
		=
		o(n^p)
		\]
		after choosing \(\eta>0\) small enough, since \(p<1/2\).  By
		Corollary~\ref{cor:backbone-susceptibility},
		\[
		\EE[A_{\rm new}\mid\mathcal H]
		=
		\rho a+o(n^p).
		\]
		Similarly,
		\[
		\begin{aligned}
			\Var(A_{\rm new}\mid\mathcal H)
			&=
			\sum_{j\notin D}W_{n,j}^2q_j(1-q_j)  \\
			&\le
			C\frac{a}{n}\sum_jW_{n,j}^3
			=
			o(n^{2p}),
		\end{aligned}
		\]
		again with \(\eta>0\) small and \(a=O(n^p)\); here we used
		\[
		q_j
		\le
		\sum_{u\in F}\hat p_{uj}
		\le
		C\frac{aW_{n,j}}n
		\]
		on the same high-probability event.  Chebyshev's inequality gives
		\[
		A_{\rm new}
		=
		\rho a+o_{\PP}(n^p).
		\]
		
		Apply this estimate to the breadth-first frontiers
		\[
		D_{n,r}^{(i)}:=\bigcup_{s=0}^rF_{n,s}^{(i)},
		\qquad
		A_{n,r}^{(i)}=\sum_{j\in F_{n,r}^{(i)}}W_{n,j}.
		\]
		Since \(A_{n,0}^{(i)}=W_{n,i}=O_{\PP}(n^p)\), induction over any fixed
		number of generations shows that the current frontier mass and explored mass
		remain \(O_{\PP}(n^p)\).  Hence, for each fixed \(r\ge1\),
		\[
		A_{n,r}^{(i)}
		=
		\rho A_{n,r-1}^{(i)}+o_{\PP}(n^p),
		\]
		and another induction yields
		\[
		A_{n,r}^{(i)}
		=
		\rho^rW_{n,i}+o_{\PP}(n^p).
		\]
		
		It remains to make the passage from a fixed number of generations to the
		whole component.  Choose \(\delta>0\) with \(\rho+\delta<1\), and choose
		\(\alpha>0\) so small that \(p+\alpha<1/2\).  Let \(\mathcal H_r\) be the
		exploration history through generation \(r\).  On the same high-probability
		events, the preceding expectation expansion, now with discovered mass at most
		\(n^{p+\alpha}\), implies the stopped recursion
		\[
		\EE[A_{n,r+1}^{(i)}\mid\mathcal H_r]
		\le
		(\rho+\delta)A_{n,r}^{(i)}
		\]
		whenever the mass discovered through generation \(r\) is at most
		\(n^{p+\alpha}\).  Indeed,
		\[
		\frac{A_{n,r}^{(i)}}n\max_jW_{n,j}
		\le
		n^{p+\alpha-1}\max_jW_{n,j}
		=
		o(1),
		\]
		and removing the discovered set changes
		\(\sum_jW_{n,j}^2/n\) by at most
		\(n^{p+\alpha}\max_jW_{n,j}/n=o(1)\).  Also, after division by
		\(A_{n,r}^{(i)}\), the second-order expectation term is bounded by
		\[
		n^{p+\alpha-2+\max\{1,3p\}+\eta}
		=
		o(1)
		\]
		for \(\eta>0\) small enough.  Thus the conditional mean is at most
		\((\rho+\delta)A_{n,r}^{(i)}\) throughout the stopped range.
		
		For the tail itself, the union-bound part of the same estimate gives a
		global recursion, on these weight events and for all generations.  Here we
		upper bound the mass of the next frontier by summing all possible hits from
		the current frontier, so collisions and already discovered blobs can only
		decrease the actual new mass:
		\[
		\EE[A_{n,r+1}^{(i)}\mid\mathcal H_r]
		\le
		(\rho+\delta)A_{n,r}^{(i)},
		\]
		because
		\[
		q_j\le\sum_{u\in F_{n,r}^{(i)}}\hat p_{uj}
		\le
		(1+o(1))\frac{\beta A_{n,r}^{(i)}W_{n,j}}{n}
		\]
		uniformly and
		\(\beta n^{-1}\sum_jW_{n,j}^2\le\rho+\delta/2\) for all large \(n\) on the
		event.  Iterating gives, on this weight event, the conditional tail bound
		\[
		\EE\left[
		\sum_{r>R}A_{n,r}^{(i)}
		\,\middle|\,(W_{n,j})_j
		\right]
		\le
		C W_{n,i}(\rho+\delta)^R,
		\]
		with \(C=(\rho+\delta)/(1-\rho-\delta)\), say.  The same recursion also gives
		\[
		\PP\left(
		\sup_r\sum_{s=0}^rA_{n,s}^{(i)}>n^{p+\alpha}
		\,\middle|\,(W_{n,j})_j
		\right)
		\le
		C W_{n,i}n^{-p-\alpha}=o(1)
		\]
		whenever \(W_{n,i}=O(n^p)\), so the exceptional large-discovery event is
		negligible.
		
		By Markov's inequality, after unconditioning and using that the weight event
		has probability tending to one,
		\[
		\sum_{r>R}A_{n,r}^{(i)}
		=
		O_{\PP}\left(W_{n,i}(\rho+\delta)^R\right)+o_{\PP}(n^p).
		\]
		For fixed \(R\),
		\[
		\sum_{r=0}^R A_{n,r}^{(i)}
		=
		W_{n,i}\sum_{r=0}^R\rho^r+o_{\PP}(n^p).
		\]
		Since \(W_{n,i}=O_{\PP}(n^p)\), first let \(n\to\infty\) and then
		\(R\to\infty\).  The geometric tails with ratios \(\rho\) and
		\(\rho+\delta\) vanish, and therefore
		\[
		|\cC_{n,i}^{\full}|
		=
		\sum_{r\ge0}A_{n,r}^{(i)}
		=
		\frac{W_{n,i}}{1-\rho}+o_{\PP}(n^p).
		\]
	\end{proof}
	
	The preceding lemma controls the shortcut cloud grown from a fixed leading
	blob.  To prove the main subcritical blob theorem, we also need a uniform
	bound showing that components seeded only by lower-ranked blobs cannot reach
	the \(n^p\)-scale of the leading blobs.  For that comparison, we first define
	a truncated branching process and record its tail estimate.
	
	\paragraph{\bf Truncated marked branching process.}
	Let \(\bar\beta>0\), let \((w_{n,i})_{i\ge1}\) be deterministic positive
	weights, and fix a cutoff \(x>0\).  Put
	\[
	I_x:=\{i:w_{n,i}\le x\}.
	\]
	For a root index \(a\in I_x\), let \(\mathcal T^{(x)}(a)\) be the following
	multitype Galton--Watson process.  The root particle has type \(a\) and mark
	\(w_{n,a}\).  Every other particle also has a type in \(I_x\); a particle of
	type \(b\) has mark \(w_{n,b}\), and gives, independently for every
	\(j\in I_x\), a Poisson number of children of type \(j\) with mean
	\[
	\frac{\bar\beta w_{n,b}w_{n,j}}{n}.
	\]
	Let \(M(v)\) be the mark of particle \(v\), and define the total progeny mass
	\[
	T^{(x)}(a)
	:=
	\sum_{v\in\mathcal T^{(x)}(a)}M(v),
	\]
	including the root contribution \(w_{n,a}\).  Thus all marks used by the
	process are entries of the weight array, and the cutoff \(x\) only restricts
	the allowed child types.
	
	\begin{lem}[Exponential rank-truncated progeny bound]
		\label{lem:marked-progeny-tail}
		Assume that, for some \(\bar\nu<1\),
		\[
		\frac{\bar\beta}{n}
		\sum_{i\in I_x}w_{n,i}^2
		\le \bar\nu.
		\]
		There exist constants \(c,\theta>0\), depending only on \(\bar\nu\), such
		that, for every \(a\in I_x\),
		\begin{equation}
			\label{eq:marked-progeny-exponential}
			\EE\exp\left\{\theta T^{(x)}(a)/x\right\}
			\le \exp\{cw_{n,a}/x\}.
		\end{equation}
		Consequently, for every \(r>0\), there is \(C_r<\infty\), depending only on
		\(r\) and \(\bar\nu\), such that whenever \(a\in I_x\) satisfies
		\(w_{n,a}=x\),
		\begin{equation}
			\label{eq:marked-progeny-polynomial}
			\PP\left(T^{(x)}(a)>t\right)
			\le C_r\left(\frac{x}{t}\right)^r,
			\qquad t>0,
		\end{equation}
		uniformly in \(n\) and in all weight arrays satisfying the displayed
		second-moment bound.
	\end{lem}
	
	\begin{proof}
		Choose \(c>0\) so small that
		\[
		\bar\nu(e^c-1)<c,
		\]
		and set
		\[
		\theta:=c-\bar\nu(e^c-1)>0.
		\]
		First restrict the process to a finite set \(J\subset I_x\) of admissible
		child types, omitting all child types outside \(J\).  For a root index
		\(a\in I_x\), let \(\mathcal T_J^{(x)}(a)\) be the resulting genealogical
		tree, write \(T_J^{(x)}(a)\) for its total progeny mass, and put
		\[
		T_{J,m}^{(x)}(a)
		:=
		\sum_{v\in\mathcal T_J^{(x)}(a):\, |v|\le m}M(v),
		\]
		where \(|v|\) is the generation of \(v\).  We prove by induction on \(m\)
		that, uniformly in the finite set \(J\),
		\begin{equation}
			\label{eq:finite-generation-mgf}
			\EE\exp\left\{\theta T_{J,m}^{(x)}(a)/x\right\}
			\le \exp\{cw_{n,a}/x\},
			\qquad a\in I_x.
		\end{equation}
		For \(m=0\), this follows from \(\theta<c\).  Suppose it holds through
		generation \(m\).  The Poisson generating functional gives
		\begin{align*}
			&\log \EE\exp\left\{\theta T_{J,m+1}^{(x)}(a)/x\right\} \\
			&\qquad=
			\frac{\theta w_{n,a}}{x}
			+\frac{\bar\beta w_{n,a}}{n}
			\sum_{j\in J}w_{n,j}
			\left(
			\EE e^{\theta T_{J,m}^{(x)}(j)/x}-1
			\right) \\
			&\qquad\le
			\frac{\theta w_{n,a}}{x}
			+\frac{\bar\beta w_{n,a}}{n}
			\sum_{j\in J}w_{n,j}
			\left(e^{cw_{n,j}/x}-1\right).
		\end{align*}
		For \(0\le u\le1\), convexity gives
		\(e^{cu}-1\le u(e^c-1)\).  Therefore the last display is at most
		\[
		\frac{w_{n,a}}{x}\left[
		\theta
		+\frac{\bar\beta}{n}
		\sum_{j\in J}w_{n,j}^2(e^c-1)
		\right]
		\le
		\frac{w_{n,a}}{x}\big[\theta+\bar\nu(e^c-1)\big]
		=\frac{cw_{n,a}}{x}.
		\]
		This proves \eqref{eq:finite-generation-mgf}.  Since
		\(T_{J,m}^{(x)}(a)\uparrow T_J^{(x)}(a)\), monotone convergence gives
		\[
		\EE\exp\left\{\theta T_J^{(x)}(a)/x\right\}
		\le \exp\{cw_{n,a}/x\}
		\]
		for every finite \(J\subset I_x\) and \(a\in I_x\), with constants
		independent of \(J\).
		Choose finite sets \(J_k\uparrow I_x\).  Under the natural coupling using the
		same independent Poisson variables,
		\[
		T_{J_k}^{(x)}(a)\uparrow T^{(x)}(a).
		\]
		Another application of monotone convergence proves
		\eqref{eq:marked-progeny-exponential} for the full countable-type process.
		If \(w_{n,a}=x\), then \eqref{eq:marked-progeny-exponential} gives
		\[
		\PP(T^{(x)}(a)>t)
		\le e^c e^{-\theta t/x}.
		\]
		For every \(r>0\), the elementary bound
		\(e^{-u}\le C_r u^{-r}\) for \(u>0\), with the constant enlarged when
		\(t<x\), gives \eqref{eq:marked-progeny-polynomial}.
	\end{proof}
	
	We next turn the preceding progeny estimate into a uniform comparison theorem
	for subcritical rank-one graphs.  The point of this abstraction is that the
	one-seed exploration above controls the component grown from a prescribed
	large blob, while the proof of the top-component ordering also needs a bound
	that holds simultaneously over all components which avoid the first few blob
	ranks.  After the leading weights are removed, every remaining component can
	be rooted at its smallest-ranked vertex and dominated by the truncated
	branching process with cutoff equal to that root weight.  The resulting union
	bound is the mechanism which prevents lower-ranked shortcut components from
	reaching the \(n^p\) scale of the leading blobs.
	
	\paragraph{\bf Rank-one comparison setup.}
	Fix \(p\in(0,1/2)\), \(\beta>0\), and \(\nu<1\).  For each \(n\), let
	\([N_n]\) be a finite vertex set and let
	\[
	w_{n,1}\ge w_{n,2}\ge\cdots\ge w_{n,N_n}>0,
	\qquad
	\sum_{i=1}^{N_n}w_{n,i}=n,
	\]
	be deterministic integer weights.  Let \(\cG_n\) be a graph on \([N_n]\)
	with independent edges, and for a component \(\cD\) of \(\cG_n\) write
	\[
	w(\cD):=\sum_{i\in\cD}w_{n,i}.
	\]
	The edge bound is the following uniform upper bound: there is
	\(\varepsilon_n\to0\) such that, whenever \(w_{n,i}w_{n,j}=o(n)\),
	\[
	\PP(i\leftrightarrow j)
	\le
	1-\exp\left\{
	-\frac{(1+\varepsilon_n)\beta w_{n,i}w_{n,j}}{n}
	\right\},
	\qquad i\ne j.
	\]
	The maximum-weight bound is
	\[
	\max_{1\le i\le N_n}w_{n,i}=O(n^p),
	\]
	and the susceptibility bound is
	\[
	\frac{\beta}{n}\sum_{i=1}^{N_n}w_{n,i}^2\le \nu
	\qquad\text{for all large }n.
	\]
	We use the convention that a maximum over an empty family of components is
	\(0\).
	
	\begin{thm}[Conditional subcritical rank-one bound]
		\label{thm:conditional-rank-one-bound}
		Assume the rank-one comparison setup satisfies the edge bound, the
		maximum-weight bound, and the susceptibility bound.  Then, for every
		\(r>0\), there is \(C_r<\infty\) such that, for every \(L\ge0\), every
		\(t>0\), and all sufficiently large \(n\),
		\begin{equation}
			\label{eq:conditional-rank-one-tail}
			\PP\left(
			\max_{\cD:\,\cD\cap[L]=\eset}w(\cD)>t
			\right)
			\le
			C_r t^{-r}\sum_{i>L}w_{n,i}^r.
		\end{equation}
		In particular, if, for a fixed \(L\) with \(L+1\le N_n\) eventually and
		some \(r_L>0\),
		\[
		\sum_{i>L}
		\left(\frac{w_{n,i}}{w_{n,L+1}}\right)^{r_L}
		\le C_L
		\qquad\text{for all large }n,
		\]
		then
		\[
		\max_{\cD:\,\cD\cap[L]=\eset}w(\cD)
		=O_{\PP}(w_{n,L+1}).
		\]
		Finally, suppose the weights are random.  Assume there are events \(E_n\),
		measurable with respect to the weights, with \(\PP(E_n)\to1\), such that on
		\(E_n\), conditional on the realized weights, the deterministic rank-one
		comparison setup above holds with the same constants \(p,\beta,\nu\) and with
		a deterministic error sequence \(\varepsilon_n\to0\) in the edge bound.  If,
		for some \(r>1/p\),
		\begin{equation}
			\label{eq:conditional-rank-one-tail-assumption}
			\lim_{L\to\infty}\limsup_{n\to\infty}
			\PP\left(
			n^{-pr}\sum_{i>L}w_{n,i}^r>\eps
			\right)=0
			\qquad\text{for every }\eps>0.
		\end{equation}
		then
		\[
		\lim_{L\to\infty}\limsup_{n\to\infty}
		\PP\left(
		n^{-p}
		\max_{\cD:\,\cD\cap[L]=\eset}w(\cD)>\eps
		\right)=0
		\]
		for every \(\eps>0\).
	\end{thm}
	
	\begin{proof}
		Choose \(\delta>0\) so small that
		\(\bar\nu:=(1+\delta)\nu<1\), and put
		\(\bar\beta=(1+\delta)\beta\).  By the maximum-weight bound,
		\[
		\max_i w_{n,i}^2=O(n^{2p})=o(n),
		\]
		so \(w_{n,i}w_{n,j}=o(n)\) uniformly over all pairs \(i\ne j\).  Let
		\(\varepsilon_n\to0\) be the uniform error in the edge bound.  For all large
		\(n\), \((1+\varepsilon_n)\beta\le\bar\beta\), and therefore
		\[
		\PP(i\leftrightarrow j)
		\le
		1-\exp\left\{-\frac{\bar\beta w_{n,i}w_{n,j}}{n}\right\},
		\qquad i\ne j.
		\]
		Since the graph edges are independent, we may couple each graph edge below
		the event that an independent Poisson variable of mean
		\[
		\frac{\bar\beta w_{n,i}w_{n,j}}{n}
		\]
		is positive.
		
		Fix \(i\), set \(S_i:=\{i,i+1,\ldots,N_n\}\), and let \(\cD_i(S_i)\) denote
		the component of \(i\) in the graph induced by \(S_i\).  Put
		\(x:=w_{n,i}\).  Because the weights are nonincreasing,
		\[
		S_i\subseteq I_x:=\{j:w_{n,j}\le x\},
		\]
		so every available target in this exploration has weight at most \(x\).  In
		the Poisson-coupled graph, an edge discovered through a positive Poisson
		variable is dominated by the corresponding Poisson number of offspring.  The
		induced exploration only uses labels in \(S_i\), while the branching process
		allows all labels in \(I_x\).  If repeated labels, collisions with already
		seen vertices, and multiple Poisson offspring are all retained as distinct
		particles, the total progeny mass can only increase.  Thus the induced
		component mass is stochastically dominated by the marked branching process of
		Lemma~\ref{lem:marked-progeny-tail} with cutoff \(x=w_{n,i}\), root index
		\(i\), and parameter \(\bar\beta\).  Its second-moment condition follows from
		the susceptibility bound:
		\[
		\frac{\bar\beta}{n}
		\sum_{j:\,w_{n,j}\le x}w_{n,j}^2
		\le
		\frac{(1+\delta)\beta}{n}\sum_{j=1}^{N_n}w_{n,j}^2
		\le \bar\nu<1.
		\]
		Consequently, for every \(r>0\),
		\begin{equation}
			\label{eq:tail-seed-component}
			\PP\left(
			w\big(\cD_i(S_i)\big)>t
			\right)
			\le C_r\left(\frac{w_{n,i}}{t}\right)^r.
		\end{equation}
		
		Every component avoiding \([L]\) has a unique smallest index \(i>L\).  For
		that index, the component is contained in the induced component of \(i\) on
		\(S_i\).  A union bound over \(i>L\) in
		\eqref{eq:tail-seed-component} proves
		\eqref{eq:conditional-rank-one-tail}.  For the fixed-\(L\) conclusion,
		apply \eqref{eq:conditional-rank-one-tail} with \(r=r_L\) and
		\(t=Aw_{n,L+1}\).  This gives
		\[
		\PP\left(
		\max_{\cD:\,\cD\cap[L]=\eset}w(\cD)
		>Aw_{n,L+1}
		\right)
		\le C_{r_L} C_L A^{-r_L},
		\]
		and the right side tends to zero as \(A\to\infty\).
		
		For the final assertion, condition on the weights on the high-probability
		event where the setup bounds hold, and set \(t=\eps n^p\) in
		\eqref{eq:conditional-rank-one-tail}.  Including the exceptional event, the
		resulting unconditional probability is at most
		\[
		o(1)
		+\EE\left[
		1\wedge
		\left(
		C_r\eps^{-r}n^{-pr}\sum_{i>L}w_{n,i}^r
		\right)
		\right].
		\]
		Put \(B_\eps:=C_r\eps^{-r}\) and
		\[
		X_{n,L}:=n^{-pr}\sum_{i>L}w_{n,i}^r.
		\]
		For every \(\eta>0\),
		\[
		\EE[1\wedge B_\eps X_{n,L}]
		\le \eta+\PP(B_\eps X_{n,L}>\eta).
		\]
		The ranked-tail assumption
		\eqref{eq:conditional-rank-one-tail-assumption} makes the probability on the
		right vanish after first \(n\to\infty\) and then \(L\to\infty\).  Since
		\(\eta>0\) is arbitrary, the desired double limit follows.
	\end{proof}
	
	\begin{lem}[Subcritical maximal component bound]
		\label{lem:subcritical-maximal-bound}
		Let \(\Gamma_n\) denote the contracted shortcut graph on blob labels.  If
		\(D\) is a component of \(\Gamma_n\), write
		\[
		\cC(D):=\bigcup_{j\in D}\cB_{n,j},
		\qquad
		|\cC(D)|=\sum_{j\in D}W_{n,j},
		\]
		for the corresponding full component.  All maxima over \(D\) below are over
		components of \(\Gamma_n\).  Then
		\[
		\lim_{L\to\infty}\limsup_{n\to\infty}
		\PP\left(
		n^{-p}
		\max_{D:\,D\cap[L]=\eset}
		|\cC(D)|>\eps
		\right)=0
		\]
		for every \(\eps>0\).
	\end{lem}
	
	\begin{proof}
		We verify the hypotheses of
		Theorem~\ref{thm:conditional-rank-one-bound} conditional on the retained tree
		blob sizes.  The total blob mass is \(n\).  On the high-probability event
		\(\max_i W_{n,i}=O(n^p)\), all products \(W_{n,i}W_{n,j}\) are \(o(n)\)
		uniformly, since \(p<1/2\).  The shortcut edge probability between blobs
		\(i\ne j\) is
		\[
		1-\left(1-\frac{p\gl}{n}+O(n^{-2})\right)^{W_{n,i}W_{n,j}},
		\]
		which satisfies the edge bound with \(\beta=p\gl\), uniformly on this event.
		The same event gives the maximum-weight bound.
		Corollary~\ref{cor:backbone-susceptibility} gives
		\[
		\frac{p\gl}{n}\sum_iW_{n,i}^2
		\probc
		\rho(p,\gl)<1,
		\]
		so, choosing deterministic \(\nu\) with \(\rho(p,\gl)<\nu<1\), the
		susceptibility bound holds with probability tending to one.
		
		Choose any \(r>1/p\).  On the good setup event, and conditionally on the blob
		sizes, apply \eqref{eq:conditional-rank-one-tail} with \(t=\eps n^p\).  Apart
		from the \(o(1)\)-probability exceptional event where the edge,
		maximum-weight, or susceptibility bound fails, this gives
		\[
		\PP\left(
		\left.
		\max_{D:\,D\cap[L]=\eset}
		|\cC(D)|>\eps n^p
		\right| (W_{n,i})_i
		\right)
		\le
		1\wedge
		\left(
		C_r\eps^{-r}n^{-pr}\sum_{i>L}W_{n,i}^r
		\right).
		\]
		The ranked-tail estimate \eqref{eq:ranked-family-r-tail} makes the quantity
		inside parentheses converge to zero in probability after first
		\(n\to\infty\) and then \(L\to\infty\).  Since the right side is bounded by
		one, taking expectations proves the displayed limit.
	\end{proof}
	
	Figure~\ref{fig:top-blob-separation-proof} depicts the extra withholding step
	used below to rule out indirect paths between leading blobs.
	
	\begin{figure}[t]
		\centering
		\begin{tikzpicture}[
			x=0.9cm,y=0.85cm,
			every node/.style={font=\small},
			topblob/.style={circle,draw=teal!75!black,fill=teal!18,thick,minimum size=29pt},
			cloud/.style={circle,draw=orange!75!black,fill=orange!15,thick,minimum size=#1},
			shortcut/.style={gray!70,densely dashed,thick}
			]
			\node[topblob,label=below:{top blob \(i\)}] (i) at (0,0) {};
			\node[topblob,label=below:{withheld top blob \(j\)}] (j) at (7.0,0.25) {};
			
			\node[cloud=17pt] (a) at (2.0,0.95) {};
			\node[cloud=13pt] (b) at (2.25,-0.15) {};
			\node[cloud=10pt] (c) at (1.8,-1.0) {};
			\node[cloud=14pt] (d) at (3.55,0.55) {};
			\node[cloud=9pt] (e) at (3.6,-0.75) {};
			\node[cloud=8pt] (f) at (4.75,0.05) {};
			
			\draw[shortcut] (i) -- (a);
			\draw[shortcut] (i) -- (b);
			\draw[shortcut] (b) -- (c);
			\draw[shortcut] (a) -- (d);
			\draw[shortcut] (b) -- (e);
			\draw[shortcut] (d) -- (f);
			
			\draw[red!70!black,densely dashed,very thick,-{Latex[length=2.4mm]}]
			(f) -- node[above,sloped] {\(W_jR_{n,i}^{(K)}/n\)} (j);
			\draw[teal!60!black,rounded corners,thick,dashed]
			(1.25,-1.35) rectangle (5.25,1.35);
			\node[teal!65!black,align=center] at (3.25,2.05)
			{lower-rank exploration cloud\\mass \(R_{n,i}^{(K)}=O_{\PP}(n^p)\)};
			\node[align=center] at (3.5,-1.75)
			{withholding rules out paths through this cloud,\\not only direct top--top shortcuts};
		\end{tikzpicture}
		\caption{Withholding step in Lemma~\ref{lem:top-blobs-separated}.  To test
			whether \(\cB_{n,i}\) can share a full shortcut component with another top
			blob, all other top blobs are withheld and the component of \(\cB_{n,i}\) is
			explored only through lower-ranked blobs.  The explored cloud has mass
			\(R_{n,i}^{(K)}=O_{\PP}(n^p)\); any connection to a withheld top blob must use
			an unqueried shortcut from that blob to the cloud, with conditional probability
			bounded by \(O_{\PP}(n^{2p-1})\).}
		\label{fig:top-blob-separation-proof}
	\end{figure}
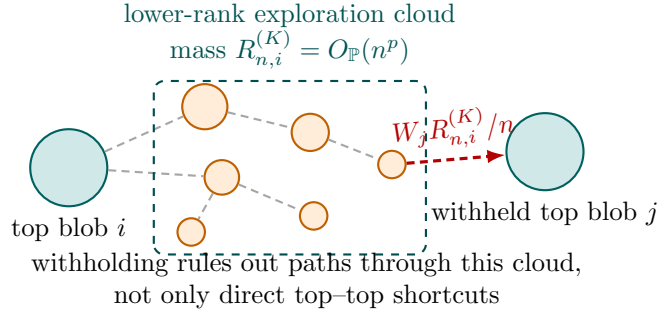
	
	\begin{lem}[Top blobs lie in distinct shortcut components]
		\label{lem:top-blobs-separated}
		For every fixed \(K\),
		\[
		\PP\left(
		\text{two of }\cB_{n,1},\ldots,\cB_{n,K}
		\text{ lie in the same full shortcut component}
		\right)
		\to0.
		\]
	\end{lem}
	
	\begin{proof}
		Fix \(i\le K\).  Delete the blobs
		\(\cB_{n,1},\ldots,\cB_{n,K}\) except for \(\cB_{n,i}\), and explore the
		shortcut component of \(\cB_{n,i}\) in the remaining rank-one blob graph.  The
		one-blob exploration proof of Lemma~\ref{lem:subcritical-exploration}, with
		fewer available target blobs, gives that the total discovered mass
		\(R_{n,i}^{(K)}\) is \(O_{\PP}(n^p)\).
		
		If, in the original full shortcut graph, \(\cB_{n,i}\) is connected to some
		other top blob, then at least one retained shortcut joins the explored set to
		one of the withheld blobs \(\cB_{n,j}\), \(j\le K\), \(j\ne i\).  Indeed, take
		a path from \(\cB_{n,i}\) to the first withheld top blob it reaches; the
		preceding vertex on the path lies in the explored set.  Conditional on the
		explored set and the blob sizes, none of the shortcut edges from the withheld
		top blobs to the explored set has been queried during the withheld
		exploration.  Hence their conditional probability of containing at least one
		retained edge is at most
		\[
		H_{n,i}^{(K)}:=
		C\frac{R_{n,i}^{(K)}}{n}
		\sum_{\substack{j\le K\\ j\ne i}}W_{n,j}.
		\]
		By Proposition~\ref{prop:rrt-blob-input}, the finite sum over \(j\le K\) is
		\(O_{\PP}(n^p)\), while \(R_{n,i}^{(K)}=O_{\PP}(n^p)\).  Thus the displayed
		random upper bound satisfies
		\(H_{n,i}^{(K)}=O_{\PP}(n^{2p-1})=o_{\PP}(1)\), since \(p<1/2\).  The
		conditional probability is at most \(1\wedge H_{n,i}^{(K)}\), so taking
		expectations gives \(o(1)\).  A union bound over the finitely many \(i\le K\)
		proves the claim.
	\end{proof}
	
	\subsection{Proof of Theorem~\ref{thm:subcritical-blob-picture}}
	\label{sec:proof-subcritical-blob-picture}
	
	\begin{proof}
		Part (a) is the root-rank tightness in
		Proposition~\ref{prop:rrt-blob-input}(b).  Indeed, after the deterministic
		tie rule has been fixed, the event \(\{\rank(\cB_n(1))>K\}\) is the event
		that \(K\) ranked blobs precede the root blob.  Since the root blob has
		smallest label \(1\), it precedes every other blob of the same size under the
		tie rule.  Hence
		\[
		\{\rank(\cB_n(1))>K\}
		=
		\{|\cB_n(1)|<W_{n,K}\}.
		\]
		The displayed tightness estimate in Proposition~\ref{prop:rrt-blob-input}(b)
		therefore lets us choose \(K\) so that the desired limsup is smaller than
		\(\eps\).
		
		Part (b) is Proposition~\ref{prop:rrt-blob-input}(c)--(d), together with the
		beta-function tail asymptotic
		\[
		B\left(k,1+\frac1p\right)
		\sim
		\Gamma\left(1+\frac1p\right)k^{-1-1/p}.
		\]
		Thus
		\[
		q_k
		\sim
		\frac{1-p}{p}\Gamma\left(1+\frac1p\right)k^{-1-1/p},
		\]
		and summing this regularly varying tail gives
		\(\sum_{\ell\ge k}q_\ell\sim
		(1-p)\Gamma(1+1/p)k^{-1/p}\).  The susceptibility identity is
		Corollary~\ref{cor:backbone-susceptibility}.
		
		We prove part (c).  Put \(\rho=\rho(p,\gl)\).  The one-seed asymptotic
		\[
		|\cC_{n,i}^{\full}|
		=
		\frac{W_{n,i}}{1-\rho}
		+
		o_{\PP}(n^p)
		\]
		for every fixed \(i\) is exactly
		Lemma~\ref{lem:subcritical-exploration}.  It remains to show that, for fixed
		\(K\), no component seeded outside the first \(K\) blobs can overtake these
		\(K\) components.
		
		Fix \(K\) and an auxiliary error level \(\eta>0\).  By
		Proposition~\ref{prop:rrt-blob-input}, the ranked vector
		\[
		n^{-p}(W_{n,1},\ldots,W_{n,K+1})
		\]
		converges almost surely to a strictly decreasing limit.
		Thus there is a deterministic \(\delta>0\) such that, with lower limiting
		probability at least \(1-\eta\),
		\[
		W_{n,K}\ge 3\delta n^p,
		\qquad
		W_{n,K}-W_{n,K+1}\ge 3\delta n^p
		\]
		after reducing \(\delta\) if necessary.  Write \(\cC(D)\) for the full
		component corresponding to a component \(D\) of the contracted shortcut graph.
		Next choose \(L\) large enough that
		Lemma~\ref{lem:subcritical-maximal-bound}, applied with cutoff \(K+L\), gives
		\[
		\limsup_{n\to\infty}
		\PP\left(
		\max_{D:\,D\cap[K+L]=\eset}
		|\cC(D)|>
		\frac{\delta}{4(1-\rho)}\,n^p
		\right)
		<\eta .
		\]
		
		For this fixed \(L\), apply Lemma~\ref{lem:subcritical-exploration} to the
		finite set of seeds \(1,\ldots,K+L\).  A union bound over this finite set
		gives, with probability tending to one,
		\[
		\max_{1\le i\le K+L}
		\left|
		|\cC_{n,i}^{\full}|
		-
		\frac{W_{n,i}}{1-\rho}
		\right|
		\le
		\frac{\delta}{4(1-\rho)}\,n^p .
		\]
		Also, by Lemma~\ref{lem:top-blobs-separated}, the first \(K\) blobs lie in
		distinct full shortcut components with probability tending to one.
		
		On the intersection of these events, the smallest of the first \(K\)
		top-seeded components has size at least
		\[
		\frac{W_{n,K}}{1-\rho}
		-
		\frac{\delta}{4(1-\rho)}n^p.
		\]
		Any component which meets one of the labels \(K+1,\ldots,K+L\) but avoids the
		first \(K\) labels is equal to \(\cC_{n,j}^{\full}\) for at least one such
		seed \(j\), and hence has size at most
		\[
		\frac{W_{n,K+1}}{1-\rho}
		+
		\frac{\delta}{4(1-\rho)}n^p.
		\]
		The gap condition makes the first of these two bounds larger than the second.
		Indeed, their difference is at least
		\[
		\frac{W_{n,K}-W_{n,K+1}}{1-\rho}
		-
		\frac{\delta}{2(1-\rho)}n^p
		\ge
		\frac{5\delta}{2(1-\rho)}n^p.
		\]
		Any component avoiding all labels in \([K+L]\) has size at most
		\(\delta n^p/[4(1-\rho)]\) on the chosen tail event, while every one of the
		first \(K\) seeded components has size at least
		\[
		\frac{W_{n,K}}{1-\rho}
		-
		\frac{\delta}{4(1-\rho)}n^p
		\ge
		\frac{11\delta}{4(1-\rho)}n^p.
		\]
		Hence no component outside the first \(K\) seeded components can enter the
		top \(K\).
		
		The probability of the comparison event is at least \(1-O(\eta)\) in the
		lower limit.  Since \(\eta>0\) is arbitrary, the top-\(K\) identification
		holds with probability tending to one, and the displayed size asymptotics
		follow from Lemma~\ref{lem:subcritical-exploration}.
	\end{proof}
	
	\begin{rem}[Janson comparison]
		Janson's theorem~\cite{Janson2008SubcriticalPowerLaw} for subcritical random
		graphs with power-law degrees says
		that the largest components are obtained by starting from the largest weights
		or degrees and attaching subcritical branching-process trees, with
		amplification factor \((1-\nu)^{-1}\). The present proof is the same mechanism
		after contracting the retained recursive-tree backbone into blobs. The role of
		the degree is played by the blob mass \(W\), and the effective offspring mean is
		\(\nu=\rho(p,\gl)=p\gl/(1-2p)\).
	\end{rem}
	
	\subsection{Network archaeology from subcritical clusters}
	\label{sec:proof-network-archaeology}
	
	We now convert the subcritical cluster structure into the root-finding
	procedure from Section~\ref{sec:model-notation}.  The proof has three
	deterministic-size ingredients.  First, Theorem~\ref{thm:subcritical-blob-picture}
	applied with \(p=q\) puts the full \(q\)-percolated root component among the
	largest components after choosing \(K\).  Second, the root blob itself is a
	uniform attachment tree.  The full component containing it, however, is not
	claimed to be a tree: shortcut edges outside the root blob may create cycles
	and shortcut paths inside the pieces attached to the root blob.  What we use
	is the weaker decorated-skeleton structure.  With high probability every
	external shortcut component that meets the root blob meets it at a single
	root-blob vertex, and there are no shortcut edges internal to the root blob.
	Thus the root blob is a uniform-attachment tree skeleton, while the attached
	shortcut pieces are contracted into vertex weights.  Third, weighted Jordan
	centrality is controlled on this skeleton, and the single-attachment property
	transfers those weighted-tree comparisons back to ordinary graph Jordan
	centrality in the full, generally non-tree, component.
	Figure~\ref{fig:decorated-root-skeleton-proof} shows this final decorated
	skeleton comparison.

	\subsubsection{Weighted Jordan centrality}
	\label{sec:proof-weighted-jordan}
	
	Let \(T\) be a finite tree and let \(x=(x_v)_{v\in T}\) be positive vertex
	weights. For \(A\subseteq T\), write \(x(A)=\sum_{u\in A}x_u\). Define the
	weighted Jordan score
	\[
	\psi_T^x(v)
	:=
	\max\{x(A):A\in\Comp(T\setminus\{v\})\}.
	\]
	For \(u,v\in T\), write \((T,u)_{v\downarrow}\) for the descendant subtree of
	\(v\) when \(T\) is rooted at \(u\).
	
	\begin{lem}[Weighted Jordan comparison]
		\label{lem:weighted-jordan-comparison}
		For adjacent vertices \(u,v\in T\),
		\begin{equation}
			\label{eq:weighted-jordan-comparison}
			\psi_T^x(u)\le \psi_T^x(v)
			\quad\Longleftrightarrow\quad
			x\big((T,v)_{u\downarrow}\big)
			\ge
			x\big((T,u)_{v\downarrow}\big).
		\end{equation}
	\end{lem}
	
	\begin{proof}
		Since \(u\) and \(v\) are adjacent, removing \(u\), the component containing
		\(v\) is \((T,u)_{v\downarrow}\). Every other component of
		\(T\setminus\{u\}\) is contained in \((T,v)_{u\downarrow}\). Similarly,
		removing \(v\), the component containing \(u\) is
		\((T,v)_{u\downarrow}\), and every other component is contained in
		\((T,u)_{v\downarrow}\).  In the following two displays, a maximum over an
		empty family is interpreted as \(0\).  Hence, with
		\[
		A:=x((T,u)_{v\downarrow}),
		\qquad
		B:=x((T,v)_{u\downarrow}),
		\]
		we have
		\[
		\psi_T^x(u)
		=
		\max\left\{
		A,
		\max_{\substack{
				D\in\Comp(T\setminus\{u\})\\
				D\subseteq (T,v)_{u\downarrow}}}x(D)
		\right\},
		\]
		while
		\[
		\psi_T^x(v)
		=
		\max\left\{
		B,
		\max_{\substack{
				D\in\Comp(T\setminus\{v\})\\
				D\subseteq (T,u)_{v\downarrow}}}x(D)
		\right\}.
		\]
		Every component contributing to the secondary maximum in \(\psi_T^x(u)\) is a
		proper subset of \((T,v)_{u\downarrow}\), unless the family is empty, and
		hence has weight at most \(B\).  Similarly, every component contributing to
		the secondary maximum in \(\psi_T^x(v)\) is a proper subset of
		\((T,u)_{v\downarrow}\), and hence has weight strictly smaller than \(A\)
		whenever the family is nonempty, because all vertex weights are positive.
		Thus \(B\ge A\) implies \(\psi_T^x(u)\le B=\psi_T^x(v)\), while \(A>B\)
		implies \(\psi_T^x(u)=A>\psi_T^x(v)\). This proves
		\eqref{eq:weighted-jordan-comparison}.
	\end{proof}
	
	\begin{rem}
		The unweighted version of this adjacent-vertex comparison is the elementary
		tree-centroid comparison used by Jog--Loh in Lemma~2.1 and its proof
		\cite{JogLoh2018}.  The present lemma is the same deterministic observation
		with cardinalities replaced by positive vertex weights.
	\end{rem}

	\begin{lem}[Robust weighted root finding for uniform attachment]
		\label{lem:weighted-ua-root-finding}
		Let \(U_m\) be a uniform attachment tree on vertices \([m]\), rooted at
		\(1\). Let \((X_v)_{v\in[m]}\) be independent positive weights, independent
		of \(U_m\), such that for some constants \(0<c<C<\infty\),
		\begin{equation}
			\label{eq:weight-moments}
			c\le \EE X_v\le C,
			\qquad
			\sup_v \EE X_v^2\le C .
		\end{equation}
		Let \(\TopJ_L^X(U_m)\) be the \(L\) vertices with smallest weighted Jordan
		scores. For \(\delta>0\), put
		\[
		\cA_\delta^X(U_m)
		:=
		\{v\in U_m:
		\psi_{U_m}^X(v)\le \psi_{U_m}^X(1)+\delta X(U_m)\}.
		\]
		Then for every \(\eta>0\), there exist
		\[
		L=L(\eta,c,C)<\infty,
		\qquad
		\delta=\delta(\eta,c,C)>0,
		\qquad
		m_0=m_0(\eta,c,C)<\infty,
		\]
		such that, for every \(m\ge m_0\),
		\begin{equation}
			\label{eq:robust-weighted-ua-root-finding}
			\PP\left(
			\psi_{U_m}^X(1)\le (1-\delta)X(U_m),
			\quad
			|\cA_\delta^X(U_m)|\le L
			\right)
			\ge 1-\eta .
		\end{equation}
		The same bound holds conditionally, uniformly in the
			environment, in the following sense.  Let \(m\ge m_0\), let \(\cG_m\) be a
			sigma-field, and let \(E_m\in\cG_m\) be an event on which, conditionally on
			\(\cG_m\): \(U_m\) has the uniform-attachment law and is independent of the
			weights, and the weights are conditionally independent, positive, and
			satisfy
		\[
		c\le \EE[X_v\mid\cG_m]\le C,
		\qquad
		\sup_{v\le m}\EE[X_v^2\mid\cG_m]\le C .
		\]
		Then, on \(E_m\),
		\begin{equation}
			\label{eq:uniform-conditional-ua}
				\PP\left(
				\psi_{U_m}^X(1)\le (1-\delta)X(U_m),
				\ \
				|\cA_\delta^X(U_m)|\le L
				\,\middle|\,\cG_m
				\right)
				\ge 1-\eta .
		\end{equation}
		The constants \(m_0,L,\delta\) depend only on
			\((\eta,c,C)\), and not on \(m\). In particular,
			if for each \(m\) such a pair \((\cG_m,E_m)\) is given with
			\(\PP(E_m)\to1\), then taking expectations in
			\eqref{eq:uniform-conditional-ua} gives
			\(\liminf_{m\to\infty}\PP(\psi_{U_m}^X(1)\le(1-\delta)X(U_m),\,
			|\cA_\delta^X(U_m)|\le L)\ge1-\eta\).
	\end{lem}
	
	\begin{proof}
		We prove the conditional assertion, for a fixed
			\(m\ge m_0\), with \(m_0\) determined at the end of the proof; the
		unconditional statement follows by taking \(\cG_m\) trivial.  Throughout the
		proof, work on \(E_m\) and condition on \(\cG_m\).  All constants below
		depend only on the fixed bounds \(c\) and \(C\), not on the realized
		environment and not on \(m\).  We write \(\PP_{\cG}\) and
		\(\EE_{\cG}\) for conditional probability and expectation given \(\cG_m\).
		
		For \(i\le m\), let
		\[
		F_i(m):=(U_m,1)_{i\downarrow},
		\qquad
		N_i(m):=|F_i(m)|,
		\qquad
		S_i(m):=X(F_i(m)),
		\]
		and write \(T_m=X(U_m)\). Thus \(F_i(m)\) is the descendant fringe subtree of
		vertex \(i\) when \(U_m\) is rooted at \(1\).
		
		We first show that late weighted fringe subtrees are uniformly small. Fix
		\(i\), and reveal the recursive tree after vertex \(i\) is born. For
		\(s\ge i\), \(N_i(s)\) evolves by
		\[
		N_i(s+1)=N_i(s)+\xi_s,
		\qquad
		\PP(\xi_s=1\mid U_s)=\frac{N_i(s)}{s}.
		\]
		Hence \(M_i(s):=N_i(s)/s\) is a martingale with mean \(1/i\). Moreover,
		\[
		\begin{aligned}
			\EE[M_i(s+1)^2\mid U_s]
			&=
			\left(1-\frac1{(s+1)^2}\right)M_i(s)^2
			+
			\frac{M_i(s)}{(s+1)^2}  \\
			&\le
			M_i(s)^2+\frac{M_i(s)}{(s+1)^2}.
		\end{aligned}
		\]
		Taking expectations and summing over \(s\ge i\) gives the uniform second
		moment bound
		\begin{equation}
			\label{eq:fringe-second-moment}
			\sup_{m\ge i}
			\EE\left[\left(\frac{N_i(m)}{m}\right)^2\right]
			\le
			\frac{C'}{i^2}
		\end{equation} for some constant $C'$.
		Given \(U_m\) and \(\cG_m\), the conditional independence and moment bounds
		imply
		\[
		\EE_{\cG}[S_i(m)^2\mid U_m]
		\le
		C_1 N_i(m)^2.
		\]
		Therefore
		\begin{equation}
			\label{eq:weighted-fringe-second-moment}
			\EE_{\cG} S_i(m)^2\le C_1\frac{m^2}{i^2}.
		\end{equation}
		Also \(cm\le \EE_{\cG}T_m\le Cm\) and
		\(\Var_{\cG}(T_m)\le Cm\), so
		\[
		\PP_{\cG}(T_m\le cm/2)\le \frac{C_2}{m}.
		\]
		Consequently, for every \(a>0\),
		\[
		\begin{aligned}
			\PP_{\cG}\left(\max_{i>L}S_i(m)>aT_m\right)
			&\le
			\PP_{\cG}(T_m\le cm/2)
			+
			\sum_{i>L}
			\PP_{\cG}\left(S_i(m)>\frac{acm}{2}\right)  \\
			&\le
			\frac{C_2}{m}+\frac{C_3}{a^2}\sum_{i>L}i^{-2}.
		\end{aligned}
		\]
		Since \(\sum_{i>L}i^{-2}\le 1/L\), this yields the explicit
			bound, valid for every \(m\ge1\), every \(L\ge1\), and every \(a>0\),
			uniformly on \(E_m\):
		\begin{equation}
			\label{eq:late-fringe-small}
				\PP_{\cG}\left(\max_{i>L}S_i(m)>aT_m\right)
				\le
				\frac{C_2}{m}+\frac{C_3}{a^2L}.
		\end{equation}
		
		Next we show that the root is separated from total mass by a deterministic
		macroscopic gap with high probability.  Let \(R_1(m)\) and \(R_2(m)\) be the
		two branches of the root corresponding to the first two children ever born to
		the root, and let \(N_j^{\Root}(m)=|R_j(m)|\).
		The branch sizes \(N_j^{\Root}(m)\) are functions of the tree
			alone, and on \(E_m\) the conditional law of \(U_m\) given \(\cG_m\) is the
			unconditional uniform-attachment law; hence, on \(E_m\), for any tree event
			the conditional and unconditional probabilities agree, and it suffices to
			bound the unconditional probability below.
		In the Yule embedding of the uniform attachment tree, the first two children
		are born at finite times almost surely and
		\[
		m^{-1}N_j^{\Root}(m)\to P_j>0,
		\qquad j=1,2,
		\]
		almost surely.  Hence, for every \(\zeta>0\), there
		exist \(\gamma=\gamma(\zeta)>0\) and
			\(m_1=m_1(\zeta)<\infty\), depending only on \(\zeta\), such that for all
			\(m\ge m_1(\zeta)\),
		\[
		\PP_{\cG}\left(
		\min\{N_1^{\Root}(m),N_2^{\Root}(m)\}\ge \gamma m
		\right)
		\ge 1-\zeta .
		\]
		On this event, put \(B_j(m)=X(R_j(m))\).  Conditional on \(U_m\) and
		\(\cG_m\),
		\[
		\EE_{\cG}[B_j(m)\mid U_m]\ge cN_j^{\Root}(m)\ge c\gamma m,
		\qquad
		\Var_{\cG}(B_j(m)\mid U_m)\le C m .
		\]
		Chebyshev's inequality gives
		\[
		\PP_{\cG}\left(B_j(m)\le c\gamma m/2\mid U_m\right)
		\le
		\frac{C_3}{m},
		\qquad j=1,2,
		\]
		on the same event.  Also
		\[
		\PP_{\cG}(T_m>2Cm)\le \frac{C_4}{m}.
		\]
		Therefore, with conditional probability at least
		\(1-\zeta-C_5/m\), where
			\(C_5:=2C_3+C_4\) and \(m\ge m_1(\zeta)\),
		\[
		\min\{B_1(m),B_2(m)\}
		\ge
		\frac{c\gamma}{4C}T_m .
		\]
		On this event, after deleting the root, every branch misses at least one of
		the two masses \(B_1(m),B_2(m)\), and hence
		\[
		\psi_{U_m}^X(1)
		\le
		T_m-\min\{B_1(m),B_2(m)\}.
		\]
		Consequently, for every \(\zeta>0\), every
			\(\delta\le c\gamma(\zeta)/(16C)\), and every \(m\ge m_1(\zeta)\), uniformly
			on \(E_m\),
		\begin{equation}
			\label{eq:root-macroscopic-gap}
				\PP_{\cG}\left(
				\psi_{U_m}^X(1)\le (1-4\delta)T_m
				\right)
				\ge
				1-\zeta-\frac{C_5}{m}.
		\end{equation}
		
		We now fix the constants.  Given \(\eta>0\), set
			\(\zeta:=\eta/4\) and
			\[
			\delta:=\min\left\{\frac18,\ \frac{c\,\gamma(\zeta)}{16C}\right\},
			\qquad
			L:=\left\lceil\frac{8C_3}{\delta^2\eta}\right\rceil,
			\qquad
			m_0:=\max\left\{m_1(\zeta),\
			\frac{8(C_2+C_5)}{\eta}\right\},
			\]
			all depending only on \((\eta,c,C)\).  Let \(m\ge m_0\).  By
			\eqref{eq:root-macroscopic-gap}, the event
			\(\{\psi_{U_m}^X(1)\le(1-4\delta)T_m\}\) fails with conditional probability
			at most \(\zeta+C_5/m\le \eta/4+\eta/8\), and by
			\eqref{eq:late-fringe-small} with \(a=\delta\), the event
			\(\{\max_{i>L}S_i(m)\le\delta T_m\}\) fails with conditional probability at
			most \(C_2/m+C_3/(\delta^2L)\le \eta/8+\eta/8\).  Hence the intersection of
			the two events has conditional probability at least \(1-\eta\), uniformly on
			\(E_m\).
		On the intersection of these two events,
		\[
		\psi_{U_m}^X(1)\le (1-4\delta)T_m
		\le
		(1-\delta)T_m,
		\]
		and every vertex \(i>L\) satisfies \(S_i(m)\le\delta T_m\). For such \(i\),
		the component of \(U_m\setminus\{i\}\) containing the root has weighted mass at
		least \(T_m-S_i(m)\), so
		\[
		\psi_{U_m}^X(i)
		\ge
		(1-\delta)T_m
		>
		(1-4\delta)T_m+\delta T_m
		\ge
		\psi_{U_m}^X(1)+\delta T_m.
		\]
		Thus no vertex \(i>L\) belongs to \(\cA_\delta^X(U_m)\), and hence
		\(|\cA_\delta^X(U_m)|\le L\). This proves
		\eqref{eq:uniform-conditional-ua} on \(E_m\), for every
			\(m\ge m_0\).  Taking \(\cG_m\) trivial gives
			\eqref{eq:robust-weighted-ua-root-finding}, and if \(\PP(E_m)\to1\), taking
			expectations gives the liminf form stated in the lemma.
	\end{proof}

		\begin{rem}\label{rem:uniform_m_bound}
			The bounds in the proof of Lemma \ref{lem:weighted-ua-root-finding} are independent $m$ given $m\geq m_0$ for sufficiently large $m_0$. This fact is useful later in the proof of Proposition~\ref{prop:root-finding-root-component}. 
	\end{rem}
	
	\subsubsection{The root component as a decorated skeleton}
	\label{sec:proof-root-decorations}
	
	Let \(\cB_n^{\Root}=\cB_n(1)\) be the \(q\)-percolated recursive-tree blob
	containing the root, and let \(\cC_n^{\Root}\) be the full \(q\)-percolated
	component containing the root after shortcuts are included.
	
	\begin{figure}[t]
		\centering
		\begin{tikzpicture}[
			x=0.9cm,y=0.85cm,
			every node/.style={font=\small},
			skel/.style={circle,draw=teal!75!black,fill=teal!17,thick,minimum size=13pt},
			deco/.style={circle,draw=orange!75!black,fill=orange!14,thick,minimum size=#1},
			edge/.style={thick},
			dashededge/.style={gray!70,densely dashed,thick}
			]
			\node[skel,label={[font=\scriptsize,xshift=-2pt]above left:{root}}] (r) at (0,0) {\(1\)};
			\node[skel] (a) at (1.2,0.75) {};
			\node[skel] (b) at (1.35,-0.65) {};
			\node[skel] (c) at (2.45,1.0) {};
			\node[skel] (d) at (2.55,0.1) {};
			\node[skel] (e) at (2.5,-1.05) {};
			\draw[edge] (r)--(a)--(c);
			\draw[edge] (a)--(d);
			\draw[edge] (r)--(b)--(e);
			
			\node[deco=12pt] (da1) at (1.25,1.65) {};
			\node[deco=9pt] (da2) at (1.85,1.95) {};
			\draw[dashededge] (a)--(da1)--(da2);
			\node[deco=15pt] (db1) at (0.6,-1.45) {};
			\node[deco=8pt] (db2) at (0.0,-1.9) {};
			\draw[dashededge] (b)--(db1)--(db2);
			\node[deco=10pt] (dd1) at (3.35,0.15) {};
			\draw[dashededge] (d)--(dd1);
			\node[deco=11pt] (de1) at (3.25,-1.55) {};
			\draw[dashededge] (e)--(de1);
			
			\draw[teal!70!black,rounded corners,thick]
			(-0.35,-1.25) rectangle (2.85,1.35);
			\node[teal!70!black,align=center,font=\scriptsize] at (1.25,-2.35)
			{uniform-attachment\\skeleton};
			
			\draw[-{Latex[length=2.8mm]},thick]
			(3.85,0.2) -- node[above,align=center,font=\scriptsize]
			{contract\\decorations} (4.95,0.2);
			
			\node[circle,draw=teal!75!black,fill=teal!17,thick,minimum size=19pt,
			label={[font=\scriptsize,yshift=-2pt]below left:{\(X_1\)}}] (wr) at (5.45,0) {};
			\node[circle,draw=teal!75!black,fill=teal!17,thick,minimum size=24pt,
			label={[font=\scriptsize]above:{\(X_a\)}}] (wa) at (6.55,0.75) {};
			\node[circle,draw=teal!75!black,fill=teal!17,thick,minimum size=27pt,
			label={[font=\scriptsize]below left:{\(X_b\)}}] (wb) at (6.7,-0.65) {};
			\node[circle,draw=teal!75!black,fill=teal!17,thick,minimum size=15pt] (wc) at (7.65,1.0) {};
			\node[circle,draw=teal!75!black,fill=teal!17,thick,minimum size=20pt] (wd) at (7.8,0.1) {};
			\node[circle,draw=teal!75!black,fill=teal!17,thick,minimum size=21pt] (we) at (7.75,-1.05) {};
			\draw[edge,teal!70!black] (wr)--(wa)--(wc);
			\draw[edge,teal!70!black] (wa)--(wd);
			\draw[edge,teal!70!black] (wr)--(wb)--(we);
			\node[align=center,font=\scriptsize,text width=3.3cm] at (6.75,-2.1)
			{same skeleton; vertex \(u\) carries mass \(X_u=|D_u|\)};
		\end{tikzpicture}
		\caption{The root component as a decorated skeleton.  The root blob is a
			uniform attachment tree; shortcut pieces attach singly to its vertices.  The
			right panel contracts each decoration \(D_u\) into the vertex weight
			\(X_u=|D_u|\), giving the weighted skeleton used for the Jordan comparison.}
		\label{fig:decorated-root-skeleton-proof}
	\end{figure}
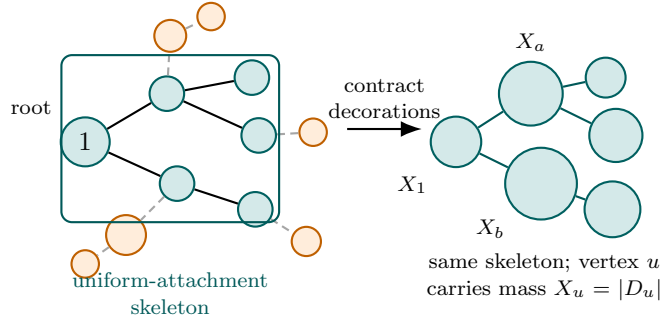
	
	\begin{lem}[Root blob genealogy]
		\label{lem:root-blob-genealogy}
		Let \(A=\cB_n^{\Root}\) denote the vertex set of the root blob.  Conditional
		on any possible realization \(A\subseteq[n]\), the parent choices of the
		vertices in \(A\setminus\{1\}\) are independent, and the parent of
		\(j\in A\setminus\{1\}\) is uniform on \(A\cap[j-1]\).  Thus, after relabeling
		\(A\) in increasing birth order, the rooted tree induced on \(A\) is a uniform
		attachment tree on \(|A|\) vertices.  In particular, the same conclusion holds
		conditional on \(|\cB_n^{\Root}|=m\).

		This conditional law remains unchanged after further conditioning on the
		parent choices and tree-percolation indicators belonging to vertices outside
		\(A\), and on shortcut variables independent of the recursive-tree
		construction, provided that the parent choices of the vertices in
		\(A\setminus\{1\}\) are not revealed.
	\end{lem}
	
	\begin{proof}
		Use the direct discrete construction of the random recursive tree.  For each
		\(j\ge2\), let \(P_j\) be the parent of \(j\), chosen uniformly from
		\([j-1]\), independently over \(j\).  Independently of all parent choices, let
		\(R_j\) be the Bernoulli-\(q\) indicator that the tree edge
		\(\{P_j,j\}\) is retained.

		Fix \(A\subseteq[n]\) with \(1\in A\).  The event that \(A\) is exactly the
		retained-tree cluster of the root is the intersection, over \(j\ge2\), of the
		following vertexwise events:
		\[
		\begin{cases}
		R_j=1\text{ and }P_j\in A\cap[j-1],
		&j\in A,\\
		\text{not both }R_j=1\text{ and }P_j\in A\cap[j-1],
		&j\notin A.
		\end{cases}
		\]
		Indeed, the first condition connects every vertex of \(A\setminus\{1\}\)
		backward through retained edges to the root, while the second prevents any
		vertex outside \(A\) from being the first vertex of a retained path entering
		\(A\).  Conversely, these conditions make \(A\) precisely the root cluster.

		Because the pairs \((P_j,R_j)\) are independent over \(j\), these conditioning
		conditions factor over the vertices.  For \(j\in A\setminus\{1\}\), conditioning
		on its factor gives
		\[
		\PP\big(P_j=i\mid R_j=1,\ P_j\in A\cap[j-1]\big)
		=
		\frac1{|A\cap[j-1]|},
		\qquad i\in A\cap[j-1].
		\]
		Consequently, conditional on the root-blob vertex set \(A\), the variables
		\((P_j:j\in A\setminus\{1\})\) remain independent with precisely these laws.
		Relabeling \(A\) in increasing order therefore gives the usual recursive-tree
		construction on \([|A|]\).  Mixing over the possible sets \(A\) of size \(m\)
		proves the assertion conditional on \(|\cB_n^{\Root}|=m\).

		Finally, after conditioning on \(A\), the factorization also separates the
		internal parent variables from all pairs \((P_j,R_j)\) with \(j\notin A\).
		Thus revealing those exterior tree variables does not alter the displayed
		internal parent laws.  The shortcut layer is independent of every
		\((P_j,R_j)\), so further revealing any specified collection of shortcut
		variables likewise leaves those laws unchanged, as long as the internal parent
		choices themselves are not revealed.
	\end{proof}

	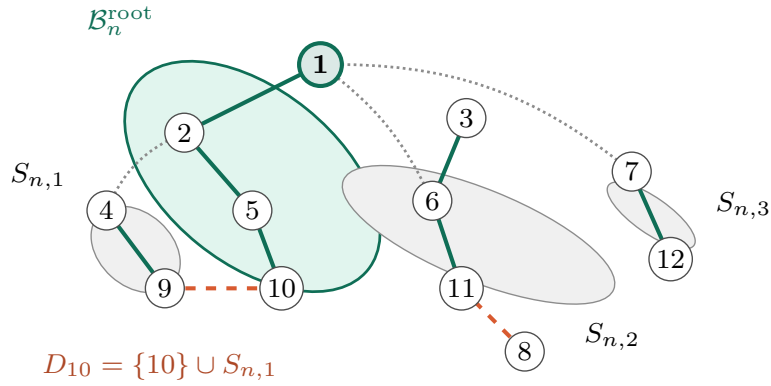
\begin{figure}[t]
		\centering
		\resizebox{0.66\textwidth}{!}{%
			\begin{tikzpicture}[scale=1.0]
				\coordinate (v1)  at (3.4,3.3);
				\coordinate (v2)  at (2.0,2.6);
				\coordinate (v3)  at (4.9,2.75);
				\coordinate (v4)  at (1.2,1.8);
				\coordinate (v5)  at (2.7,1.8);
				\coordinate (v6)  at (4.55,1.9);
				\coordinate (v7)  at (6.6,2.2);
				\coordinate (v8)  at (5.5,0.35);
				\coordinate (v9)  at (1.8,1.0);
				\coordinate (v10) at (3.0,1.0);
				\coordinate (v11) at (4.85,1.0);
				\coordinate (v12) at (7.0,1.3);
				\begin{scope}[on background layer]
					\node[ellipse,fill=blobfill,draw=rootcol,thick,inner sep=-2pt,rotate=51,
					fit=(v1)(v2)(v5)(v10)] {};
					\node[ellipse,fill=black!6,draw=black!45,inner sep=-1pt,rotate=48,fit=(v4)(v9)] {};
					\node[ellipse,fill=black!6,draw=black!45,inner sep=-4pt,rotate=68,fit=(v3)(v6)(v11)(v8)] {};
					\node[ellipse,fill=black!6,draw=black!45,inner sep=-2pt,rotate=55,fit=(v7)(v12)] {};
				\end{scope}
				\foreach \a/\b in {1/2,2/5,5/10,3/6,6/11,4/9,7/12} \draw[treecol,very thick] (v\a) -- (v\b);
				\draw[shortcol,very thick,dashed] (v8) -- (v11);
				\draw[shortcol,very thick,dashed] (v9) -- (v10);
				\draw[black!45,densely dotted,thick] (v4) to[bend left=22] (v2);
				\draw[black!45,densely dotted,thick] (v6) to[bend right=14] (v1);
				\draw[black!45,densely dotted,thick] (v7) to[bend right=20] (v1);
				\drawnodes
				\node[font=\scriptsize,rootcol] at (1.35,3.75) {$\mathcal{B}_n^{\mathrm{root}}$};
				\node[font=\scriptsize] at (0.5,2.15) {$S_{n,1}$};
				\node[font=\scriptsize] at (6.4,0.5) {$S_{n,2}$};
				\node[font=\scriptsize] at (7.75,1.85) {$S_{n,3}$};
				\node[font=\scriptsize,shortcol!80!black] at (1.75,0.2) {$D_{10}=\{10\}\cup S_{n,1}$};
			\end{tikzpicture}%
		}
		\caption{The decorated skeleton, in the example of
			Figure~\ref{fig:percolation-pipeline}. Deleting
			$\mathcal{B}_n^{\mathrm{root}}$ and its incident retained shortcuts
			leaves the external components $S_{n,1}=2$, $S_{n,2}=4$, $S_{n,3}=2$;
			each is in general a union of blobs, not a single blob. The retained
			shortcut $\{9,10\}$ (dashed) attaches $S_{n,1}$ at the single vertex
			$10$, giving the decoration $D_{10}$ with weight $X_{10}=3$; all other
			skeleton vertices have $X_u=1$. Dotted arcs mark unrevealed
			vertex--component pairs, whose indicators define the raw weights
			$\widetilde X_u=1+\sum_a S_{n,a}I_{u,a}$; on the single-attachment
			event of Lemma~\ref{lem:singly-attached-decorations},
			$\widetilde X_u=X_u$ for all $u$.}
		\label{fig:decorations-raw-weights}
	\end{figure}

	\begin{lem}[Second and third moment susceptibility for external components]
		\label{lem:external-third-susceptibility}
			Assume  \eqref{eq:arch-q-range}. Let $G_n^{(q)}$ denote the $q$-percolated
			observed graph, i.e.\ the graph on $[n]$ whose edge set consists of all
			retained edges of $G_n$, both tree and shortcut. Delete from $G_n^{(q)}$ the
			vertices of the root blob $\cB_n^{\Root}$ together with all retained shortcut
			edges having at least one endpoint in $\cB_n^{\Root}$, and call the connected
			components of the remaining graph the \emph{external components}. Enumerate
			them by an index $a$, and let
			\[
			S_{n,a} \;:=\; |\{\text{vertices of the $a$-th external component}\}|
			\]
			be their vertex masses. (Each external component is, in general, a union of
			several retained tree blobs joined by retained shortcut edges; it is not a
			single blob. The external components are precisely the candidate decorations
			of Lemma~\ref{lem:singly-attached-decorations}: an external component that is joined to
			$\cB_n^{\Root}$ by a retained shortcut edge in $G_n^{(q)}$ becomes part of a
			decoration $D_u$.) Then
			\begin{equation}
				\label{eq:S-moments}
				\frac1n \sum_a S_{n,a}^2 \;=\; O_{\PP}(1),
				\qquad
				\frac1n \sum_a S_{n,a}^3 \;=\; O_{\PP}(1).
		\end{equation}
	\end{lem}
	
	\begin{proof}
		It is enough to prove the same bounds for the component sizes
		\((\bar S_{n,b})_b\) of the full \(q\)-percolated graph on all vertices,
		because deleting the root blob and incident shortcut edges only partitions
		components and removes vertices. Thus, for \(r\ge1\),
		\[
		\sum_a S_{n,a}^r\le \sum_b \bar S_{n,b}^r .
		\]
		Let \(V_n\) be a uniformly chosen vertex. Conditional on the full
		\(q\)-percolated graph,
		\begin{equation}
			\label{eq:uniform-component-moments}
			\EE\big[|\cC_n(V_n)|\mid G_n^{(q)}\big]
			=
			\frac1n\sum_b \bar S_{n,b}^2,
			\qquad
			\EE\big[|\cC_n(V_n)|^2\mid G_n^{(q)}\big]
			=
			\frac1n\sum_b \bar S_{n,b}^3 .
		\end{equation}
		Hence it suffices to bound the first two moments of the component of a
		uniform vertex.
		
		Condition on the \(q\)-percolated recursive-tree blobs and their masses
		\((W_{n,i})_i\).  In the contracted shortcut graph the mass of a contracted
		vertex is its blob size.  The blob containing the uniform vertex \(V_n\) has
		the empirical size-biased law
		\[
		\PP_n(W_n^\star=w)
		=
		\frac1n\sum_i W_{n,i}\ind\{W_{n,i}=w\}.
		\]
		The same size bias appears for children in the exploration: a shortcut from a
		blob is generated by choosing an endpoint among original vertices, so the
		mass distribution of a newly hit blob is proportional to its mass.
		Corollary~\ref{cor:backbone-susceptibility} and
		Lemma~\ref{lem:fixed-size-family-counts}, applied with \(p=q\), give
		\begin{equation}
			\label{eq:size-biased-second-moment}
			\EE_n W_n^\star
			=
			\frac1n\sum_i W_{n,i}^2
			\probc
			\frac1{1-2q},
			\qquad
			\EE_n[(W_n^\star)^2]
			=
			\frac1n\sum_i W_{n,i}^3
			=
			O_{\PP}(1),
		\end{equation}
		where the second bound is precisely the finite-second-moment bound for the
		size-biased blob law and uses \(q<1/3\), through
		\eqref{eq:uniform-family-moments} with \(r=3\).  Also
		\(\max_iW_{n,i}=O_{\PP}(n^q)=o_{\PP}(n^{1/2})\).
		
		Choose \(\eps>0\) and \(\bar\rho<1\) so that
		\[
		(1+\eps)q\gl(1-2q)^{-1}<\bar\rho<1 .
		\]
		Since \(\max_iW_{n,i}=o_{\PP}(n^{1/2})\), fix a deterministic sequence
		\(a_n\downarrow0\) such that
		\(\PP(\max_iW_{n,i}\le a_nn^{1/2})\to1\).
		Fix an auxiliary error \(\zeta>0\).  Tightness in
		\eqref{eq:size-biased-second-moment}, together with the convergence of
		\(\EE_nW_n^\star\), permits us to choose a finite deterministic
		\(B_\zeta\) such that the event
		\[
		\begin{split}
		E_n(B_\zeta):=\big\{&
		\EE_n W_n^\star\le B_\zeta,\quad
		\EE_n[(W_n^\star)^2]\le B_\zeta,\quad
		\max_iW_{n,i}\le a_nn^{1/2},\\
		&\rho_n:=(1+\eps)q\gl\,\EE_n W_n^\star\le\bar\rho
		\big\}
		\end{split}
		\]
		satisfies
		\begin{equation}
			\label{eq:external-good-event-tightness}
			\limsup_{n\to\infty}\PP\big(E_n(B_\zeta)^c\big)\le\zeta.
		\end{equation}
		Here the last two conditions have probability tending to one, and the possible
		\(\zeta\)-loss is needed only for the tight second moment of the size-biased
		blob law.  On \(E_n(B_\zeta)\), all products
		\(W_{n,i}W_{n,j}\le a_n^2n=o(n)\)
		uniformly, and hence every actual shortcut edge probability between blobs
		of masses \(x\) and \(w\) is at most
		\[
		(1+\eps)\frac{q\gl xw}{n}
		\]
		for all large \(n\).
		
		On \(E_n(B_\zeta)\), the contracted shortcut exploration of the component of \(V_n\)
		is dominated by the following branching process.  A particle of mass \(w\)
		represents an explored blob of size \(w\).  Independently for each possible
		target blob \(j\), it has a Poisson number of children of mass \(W_{n,j}\)
		with mean \((1+\eps)q\gl\,wW_{n,j}/n\).  Equivalently, the total number of
		children is Poisson with mean \((1+\eps)q\gl w\), and their masses are
		independent with law \(W_n^\star\).  Its mean mass reproduction factor is
		\(\rho_n\le\bar\rho<1\), so it is subcritical on \(E_n(B_\zeta)\).  This ideal process
		dominates the actual breadth-first search because Bernoulli shortcut
		indicators are dominated by the corresponding Poisson variables, while
		repeated labels, multiple hits of the same label, edges to already seen
		labels, and the suppression of already explored labels can only decrease the
		actual discovered component mass.
		
		Let \(T(w)\) be the total mass of this branching process started from one
		particle of mass \(w\), and put \(M_1=\EE_n T(W_n^\star)\),
		\(M_2=\EE_n T(W_n^\star)^2\).  With the above \(\rho_n\),
		\[
		M_1=\frac{\EE_n W_n^\star}{1-\rho_n}.
		\]
		Indeed, if \(a=(1+\eps)q\gl\), then
		\(\EE_n T(w)=w+awM_1\), and averaging over \(w=W_n^\star\) gives this
		identity.  The same first-generation decomposition gives the needed second
		moment recursion.  Conditional on the first particle having mass \(w\), the
			offspring contribution is  \(T(w)=w+\sum_{i\le K}T_i\), where
			\(K\sim\Poi(aw)\) is the number of children of the root particle and \(T_i\)
			is the total progeny mass of the \(i\)-th child, the \(T_i\) are i.i.d.\ with
			the law of \(T(W_n^\star)\), independent of \(K\), by the branching property.
			The compound-Poisson moment identities
			\(\EE_n\sum_{i\le K}T_i=awM_1\) and
			\(\EE_n\big(\sum_{i\le K}T_i\big)^2=awM_2+(awM_1)^2\) then give
			\[
			\EE_n[T(w)^2]
			=
			w^2+2aw^2M_1+(awM_1)^2+awM_2 .
			\]
		Averaging over \(w=W_n^\star\) gives
		\[
		M_2
		\le
		C\EE_n[(W_n^\star)^2]+\rho_n M_2,
		\]
		and therefore \(M_2\le C_\zeta\) on \(E_n(B_\zeta)\), where
		\(C_\zeta<\infty\) is deterministic and may depend on
		\(B_\zeta,\bar\rho,q,\gl\), but not on \(n\).  Enlarging \(C_\zeta\) if
		necessary also bounds \(M_1\).  The event \(E_n(B_\zeta)\) is measurable with
		respect to the blob masses.  Thus, on \(E_n(B_\zeta)\), conditioning on the blob masses
		and then averaging over the shortcut graph and \(V_n\) gives
		\[
		\EE\big[|\cC_n(V_n)|\mid (W_{n,i})_i\big]\le C_\zeta,
		\qquad
		\EE\big[|\cC_n(V_n)|^2\mid (W_{n,i})_i\big]\le C_\zeta .
		\]
		Writing the expectation over the independent uniform vertex explicitly, the
		identities in \eqref{eq:uniform-component-moments} imply that, for
		\[
		Y_{n,2}:=\frac1n\sum_b\bar S_{n,b}^2,\qquad
		Y_{n,3}:=\frac1n\sum_b\bar S_{n,b}^3,
		\]
		we have
		\[
		\EE[Y_{n,2}\ind_{E_n(B_\zeta)}]\le C_\zeta,
		\qquad
		\EE[Y_{n,3}\ind_{E_n(B_\zeta)}]\le C_\zeta.
		\]
		Therefore, for \(r=2,3\) and \(M>0\),
		\[
		\limsup_{n\to\infty}\PP(Y_{n,r}>M)
		\le
		\zeta+\frac{C_\zeta}{M}.
		\]
		For each fixed \(\zeta\), let \(M\to\infty\), and then let
		\(\zeta\downarrow0\).  Equivalently, the order of limits is first
		\(n\to\infty\), then \(M\to\infty\), and finally
		\(\zeta\downarrow0\).  This proves the two \(O_{\PP}(1)\) bounds in
		\eqref{eq:S-moments} without requiring a fixed deterministic moment bound on an
		event whose probability tends to one.
	\end{proof}
	
	\begin{rem}
		The condition \(q<1/3\) enters the proof only through the third empirical
		moment of the retained-tree blob sizes.  Indeed,
		\eqref{eq:uniform-family-moments} with \(p=q\) and \(r=3\) gives
		\[
		\frac1n\sum_i W_{n,i}^3=O_{\PP}(1)
		\qquad\text{precisely in the regime }3q<1
		\]
		up to the logarithmic boundary case.  Equivalently, this is the finite
		second-moment condition for the size-biased blob law \(W_n^\star\), and it is
		used above to bound the second moment of the dominating branching-process
		total mass.  The subcritical domination itself uses the separate condition
		\(q\gl/(1-2q)<1\), namely \(q<1/(\gl+2)\).  Later appearances of the
		\(q<1/3\) restriction are inherited from the third-susceptibility estimate
		\(\frac1n\sum_a S_{n,a}^3=O_{\PP}(1)\), in particular in the second-moment
		bound for the decoration variables.
	\end{rem}
	
	\begin{lem}[Singly attached decorations]
		\label{lem:singly-attached-decorations}
		Assume \eqref{eq:arch-q-range}. With probability tending to one,
		\(\cC_n^{\Root}\) can be obtained from the skeleton
		\(\cB_n^{\Root}\) by attaching to each \(u\in\cB_n^{\Root}\) a finite connected
		graph \(D_u\), with \(D_u\cap\cB_n^{\Root}=\{u\}\), and with no connected
		decoration attached to two distinct skeleton vertices.
	\end{lem}
	
	\begin{proof}
		Contract the retained recursive-tree blobs. Conditional on the blob sizes, the
		retained shortcut graph is a rank-one graph with edge probabilities
		\[
		1-\exp\left\{-\frac{q\gl W_iW_j}{n}+o(W_iW_j/n)\right\}.
		\]
		We first record the size scale of the root full component.  Let \(R_n\) be
		the rank of the root blob among the retained tree blobs.  By
		Proposition~\ref{prop:rrt-blob-input}, \(R_n\) is tight and
		\[
		n^{-q}|\cB_n^{\Root}|\to Z_1>0
		\qquad\text{a.s.}
		\]
		On the event \(\{R_n\le K\}\), Theorem~\ref{thm:subcritical-blob-picture},
		applied with \(p=q\), identifies the component of the root blob with one of
		the finitely many components seeded by the first \(K\) blobs.  The fixed-\(K\)
		component asymptotics in that theorem give its total vertex mass
		\[
		|\cC_n^{\Root}|=O_{\PP}(n^q)=O_{\PP}(|\cB_n^{\Root}|),
		\]
		after first choosing \(K\) large enough and then using the positive root-blob
		limit.  In particular \(|\cB_n^{\Root}|=O_{\PP}(n^q)\).
		
		There are now two bad events to exclude.  The first is the existence of a
		retained shortcut edge whose two endpoints both lie in the root blob.  Such
		an edge would add a shortcut inside the skeleton.  Conditional on the root
		blob, the number of these shortcut edges has expectation
		\[
		O_{\PP}\left(\frac{|\cB_n^{\Root}|^2}{n}\right)
		=
		O_{\PP}(n^{2q-1})=o_{\PP}(1).
		\]
		Hence this first bad event has probability \(o(1)\) by Markov's inequality.
		
		The second bad event is that one external component attaches to two different
		root-blob vertices.  Delete the root blob and all shortcut edges incident to
		it, and let \(S_{n,a}\) be the resulting external component masses.  Conditional
		on these external components and on the root blob, the probability that a
		fixed root-blob vertex has at least one retained shortcut edge to component
		\(a\) is at most \(CS_{n,a}/n\).  Therefore the expected number of triples
		\((a,u,v)\), with \(u\ne v\in\cB_n^{\Root}\), for which component \(a\)
		touches both \(u\) and \(v\), is bounded by
		\[
		C\frac{|\cB_n^{\Root}|^2}{n^2}
		\sum_a S_{n,a}^2
		=
		O_{\PP}(n^{2q-1})=o_{\PP}(1),
		\]
		where \(\sum_aS_{n,a}^2=O_{\PP}(n)\) follows from
		Lemma~\ref{lem:external-third-susceptibility}.  Markov's inequality excludes
		this second bad event with probability tending to one.
		
		On the complement of these two bad events, the root blob keeps its original
		tree structure as the skeleton.  Every external component that meets the root
		component has a unique attachment vertex \(u\in\cB_n^{\Root}\); if several
		external components attach to the same \(u\), their union together with \(u\)
		is the decoration \(D_u\).  These decorations may contain shortcut edges and
		cycles internally, but each intersects the skeleton only at \(u\), and no
		connected decoration is attached to two distinct skeleton vertices.
	\end{proof}
	
	Recall the external components of Lemma~\ref{lem:external-third-susceptibility}:
		the connected components left after deleting from \(G_n^{(q)}\) the vertices of
		\(\cB_n^{\Root}\) and all retained shortcut edges with at least one endpoint in
		\(\cB_n^{\Root}\).  Use the parent variables \((P_j)_{j\ge2}\) and
		tree-retention indicators \((R_j)_{j\ge2}\) from the proof of
		Lemma~\ref{lem:root-blob-genealogy}, and write \(Q_{x,y}\) for the indicator
		that the shortcut-layer edge \(\{x,y\}\) is present after \(q\)-percolation.
		Put \(A=\cB_n^{\Root}\), and let \(\cF_n\) be the sigma-field generated by
		\begin{enumeratei}
			\item the vertex set \(A\), and hence \(m=|A|\) and the birth order inherited
			from the original labels;
			\item all pairs \((P_j,R_j)\) with \(j\notin A\); and
			\item all retained-shortcut indicators \(Q_{x,y}\) with \(x,y\notin A\).
		\end{enumeratei}
		The retained tree edges with both endpoints outside \(A\), the retained
		shortcut graph on \(A^c\), and hence the external graph and its component
		partition are all \(\cF_n\)-measurable.  In particular, the vertex set and the
		mass \(S_{n,a}\) of every external component \(a\) are
		\(\cF_n\)-measurable.

		Two families of variables are deliberately not revealed by \(\cF_n\): the
		internal parent choices \((P_j:j\in A\setminus\{1\})\), and the shortcut
		indicators \((Q_{u,x}:u\in A,x\notin A)\) between the root blob and the
		exterior.  The construction and Lemma~\ref{lem:root-blob-genealogy} therefore
		give the following three conditional properties:
		\begin{enumeratei}
			\item every external component and its mass \(S_{n,a}\) are
			\(\cF_n\)-measurable;
			\item conditional on \(\cF_n\), after relabeling \(A\) in birth order, the
			unrevealed internal genealogy is a uniform attachment tree; and
			\item conditional on \(\cF_n\), the shortcut families
			\((Q_{u,x}:x\in a)\), indexed by pairs \((u,a)\), are disjoint families of
			independent variables and are independent of the internal genealogy.
		\end{enumeratei}
		For \(u\in\cB_n^{\Root}\), define the raw mass
	\[
	\widetilde X_u
	=
	1+\sum_a S_{n,a}I_{u,a},
	\]
	where \(I_{u,a}\) is the indicator that at least one unrevealed retained
	shortcut edge joins \(u\) to the external component \(a\).  This raw
	definition allows the same external component to be counted for more than one
	root-blob vertex.
	
	\begin{lem}[Decoration moments]
		\label{lem:decoration-moments}
		Conditional on \(\cF_n\), the variables
		\((\widetilde X_u:u\in\cB_n^{\Root})\) are independent.  Moreover, they are
		conditionally independent of the internal root-blob genealogy.  On the event
		of Lemma~\ref{lem:singly-attached-decorations}, put
		\[
		X_u:=|D_u|,\qquad u\in\cB_n^{\Root},
		\]
		where \(D_u\) includes the skeleton vertex \(u\).  Off this event set
		\(X_u=1\) for definiteness; this convention only affects an \(o(1)\) event.
		Then
		\begin{equation}
			\label{eq:decoration-coupling}
			\PP\left(
			X_u=\widetilde X_u\text{ for all }u\in\cB_n^{\Root}
			\right)\to1.
		\end{equation}
		Moreover, for every \(\zeta>0\), there exists a finite deterministic
		\(C_\zeta>1\) such that the \(\cF_n\)-measurable event
		\[
		E_n^{\mathrm{mom}}(C_\zeta)
		:=
		\left\{
		1\le \EE[\widetilde X_u\mid\cF_n]\le C_\zeta
		\text{ and }
		\EE[\widetilde X_u^2\mid\cF_n]\le C_\zeta
		\text{ for all }u\in\cB_n^{\Root}
		\right\}
		\]
		satisfies
		\begin{equation}
			\label{eq:decoration-moments}
			\limsup_{n\to\infty}
			\PP\left(E_n^{\mathrm{mom}}(C_\zeta)^c\right)
			\le\zeta .
		\end{equation}
		In addition,
		\begin{equation}
			\label{eq:max-decoration-small}
			\max_{u\in\cB_n^{\Root}} X_u=o_{\PP}(|\cB_n^{\Root}|).
		\end{equation}
	\end{lem}
	
	\begin{proof}
		By property (iii) of the construction of \(\cF_n\), conditional on
		\(\cF_n\), the collections \((Q_{u,x}:x\in a)\), as \((u,a)\) varies, are
		disjoint collections of independent retained-shortcut variables.  Each
		\(I_{u,a}\) is a function only of the corresponding collection.  Hence the
		indicators \((I_{u,a})_{u,a}\), and in particular the variables
		\((\widetilde X_u:u\in\cB_n^{\Root})\), are conditionally independent.  By
		properties (ii)--(iii), their conditional joint law is also independent of the
		unrevealed internal root-blob genealogy, whose conditional law is the uniform
		attachment law by Lemma~\ref{lem:root-blob-genealogy}.
		
		The success probabilities satisfy, uniformly in \(u\) and \(a\),
		\begin{equation}
			\label{eq:attachment-probability}
			\pi_{n,u,a}:=\PP(I_{u,a}=1\mid\cF_n)
			\le
			C\frac{S_{n,a}}n .
		\end{equation}
		The raw variables intentionally do not enforce single attachment: the same
		external component \(a\) may contribute to several different
		\(\widetilde X_u\)'s.  Actual reuse is excluded only on the high-probability
		event from Lemma~\ref{lem:singly-attached-decorations}.
		On the high-probability event of Lemma~\ref{lem:singly-attached-decorations},
		there are no shortcut edges inside the root blob and no external component is
		attached to two distinct root-blob vertices. On this event, the actual
		decoration mass \(X_u\) equals the raw mass \(\widetilde X_u\) for every
		\(u\in\cB_n^{\Root}\). This proves the coupling
		\eqref{eq:decoration-coupling}.
		
		The conditional means need not be identical in \(u\), but they are uniformly
		bounded. Indeed,
		\[
		\mu_{n,u}:=\EE[\widetilde X_u\mid\cF_n]
		=
		1+\sum_a S_{n,a}\pi_{n,u,a}
		\le
		1+\frac{C}{n}\sum_aS_{n,a}^2,
		\]
		and \(\mu_{n,u}\ge1\). For the second moment, conditional
		independence and \eqref{eq:attachment-probability} give
		\[
		\begin{aligned}
			\EE[\widetilde X_u^2\mid\cF_n]
			&\le
			C\left[
			1+
			\sum_a S_{n,a}^2\pi_{n,u,a}
			+
			\left(\sum_aS_{n,a}\pi_{n,u,a}\right)^2
			\right] \\
			&\le
			C\left[
			1+
			\frac1n\sum_aS_{n,a}^3
			+
			\left(\frac1n\sum_aS_{n,a}^2\right)^2
			\right].
		\end{aligned}
		\]
		The right-hand sides are bounded, uniformly over \(u\), by a common
		\(\cF_n\)-measurable envelope of the form
		\[
		C\left[
		1+\frac1n\sum_aS_{n,a}^3
		+\left(\frac1n\sum_aS_{n,a}^2\right)^2
		\right].
		\]
		This envelope is \(O_{\PP}(1)\) by
		Lemma~\ref{lem:external-third-susceptibility}.  Since
		\(\widetilde X_u\ge1\), the lower conditional-mean bound is deterministic.
		Thus, for every \(\zeta>0\), tightness of the common envelope permits a finite
		deterministic \(C_\zeta>1\) for which
		\eqref{eq:decoration-moments} holds.
		
		Finally, put \(m=|\cB_n^{\Root}|\).  Fix \(\eps>0\) and an auxiliary
		error \(\zeta>0\), and choose \(C_\zeta\) as above.  On
		\(E_n^{\mathrm{mom}}(C_\zeta)\), Markov's inequality applied to second
		moments and a union bound give
		\[
		\PP\left(\max_{u\in\cB_n^{\Root}}\widetilde X_u>\eps m\mid\cF_n\right)
		\le
		\frac{mC_\zeta}{\eps^2m^2}
		=
		\frac{C_\zeta}{\eps^2m}.
		\]
		Thus, for every fixed \(M\),
		\[
		\PP\left(\max_{u\in\cB_n^{\Root}}\widetilde X_u>\eps m\right)
		\le
		\PP\left(E_n^{\mathrm{mom}}(C_\zeta)^c\right)
		+\PP(m\le M)+\frac{C_\zeta}{\eps^2M}.
		\]
		Since \(m\to\infty\) in probability, first letting \(n\to\infty\), then
		\(M\to\infty\), and finally \(\zeta\downarrow0\) gives
		\(\max_u\widetilde X_u=o_{\PP}(m)\).  The coupling
		\eqref{eq:decoration-coupling} then gives \eqref{eq:max-decoration-small}.
	\end{proof}
	
	The next lemma is purely deterministic.  It does not assert that
	\(\cC_n^{\Root}\) is a tree.  The tree is only the skeleton
	\(\cB_n^{\Root}\).  Since every decoration has a single attachment point,
	deleting a skeleton vertex separates the full graph according to the branches
	of the skeleton, up to the pieces lying inside the deleted vertex's own
	decoration.  Thus weighted branch masses on the skeleton approximate ordinary
	graph-Jordan branch sizes in the full decorated component.
	
	\begin{lem}[Graph Jordan is approximated by weighted Jordan]
		\label{lem:graph-to-weighted-jordan}
		Suppose the conclusion of Lemma~\ref{lem:singly-attached-decorations} holds.
		Write \(X_u:=|D_u|\) for \(u\in\cB_n^{\Root}\), and
		\(X(A):=\sum_{u\in A}X_u\).  Put
		\[
		\Delta_n:=\max_{u\in\cB_n^{\Root}}|D_u|.
		\]
		Then for every skeleton vertex \(v\in\cB_n^{\Root}\),
		\begin{equation}
			\label{eq:graph-weighted-error}
			0
			\le
			\psi_{\cC_n^{\Root}}(v)
			-
			\psi_{\cB_n^{\Root}}^X(v)
			\le
			\Delta_n .
		\end{equation}
		Moreover, fix a skeleton vertex \(v_{\mathrm{skel}}\), an integer \(L\ge1\), and
		\(\delta>0\).
		If
		\begin{equation}
			\label{eq:graph-transfer-gap}
			\Delta_n\le \frac{\delta}{4}X(\cB_n^{\Root}),
			\qquad
			\psi_{\cB_n^{\Root}}^X(v_{\mathrm{skel}})
			\le
			(1-\delta)X(\cB_n^{\Root}),
		\end{equation}
		and
		\begin{equation}
			\label{eq:graph-transfer-count}
			\#\left\{
			v\in\cB_n^{\Root}:
			\psi_{\cB_n^{\Root}}^X(v)
			\le
			\psi_{\cB_n^{\Root}}^X(v_{\mathrm{skel}})+\delta X(\cB_n^{\Root})
			\right\}
			\le L,
		\end{equation}
		then \(v_{\mathrm{skel}}\in \TopJ_L(\cC_n^{\Root})\).
	\end{lem}
	
	\begin{proof}
		For a skeleton vertex \(v\), deleting \(v\) from \(\cC_n^{\Root}\) separates
		the full graph along the components of
		\(\cB_n^{\Root}\setminus\{v\}\).  For each such skeleton branch \(A\), the
		corresponding graph component has exactly the vertices in the decorations
		\(\{D_u:u\in A\}\), and therefore has size \(X(A)\).  This exact matching is
		where the single-attachment property is used: decorations may have internal
		shortcut edges, but they do not connect two distinct skeleton branches.  The
		only remaining graph components are contained in \(D_v\setminus\{v\}\).  Let
		\(M_v\) be the largest of their sizes, with \(M_v=0\) if there is no such
		component.  Then \(M_v\le |D_v|\le\Delta_n\), and
		\[
		\psi_{\cC_n^{\Root}}(v)
		=
		\max\left\{
		\psi_{\cB_n^{\Root}}^X(v),M_v
		\right\}.
		\]
		This proves \eqref{eq:graph-weighted-error}.
		
		Now let \(z\) be a non-skeleton vertex in the decoration attached at \(u\). On
		deleting \(z\), the component containing the attachment vertex \(u\) contains
		the whole skeleton and every decoration \(D_w\) with \(w\ne u\).  The only
		vertices it can miss lie inside \(D_u\), and their number is at most
		\(|D_u|\le\Delta_n\).  Therefore
		\[
		\psi_{\cC_n^{\Root}}(z)
		\ge
		X(\cB_n^{\Root})-\Delta_n .
		\]
		This is the reason non-skeleton vertices cannot beat the root once the
		largest decoration is negligible compared with the total weighted skeleton
		mass.
		For the deterministic transfer statement, assume
		\eqref{eq:graph-transfer-gap}--\eqref{eq:graph-transfer-count}. The graph
		score of \(v_{\mathrm{skel}}\) is at most
		\[
		\psi_{\cC_n^{\Root}}(v_{\mathrm{skel}})
		\le
		\psi_{\cB_n^{\Root}}^X(v_{\mathrm{skel}})+\Delta_n
		\le
		\left(1-\frac{3\delta}{4}\right)X(\cB_n^{\Root}).
		\]
		Every non-skeleton vertex has graph score at least
		\[
		X(\cB_n^{\Root})-\Delta_n
		\ge
		\left(1-\frac{\delta}{4}\right)X(\cB_n^{\Root}),
		\]
		so no non-skeleton vertex has graph score at most
		\(\psi_{\cC_n^{\Root}}(v_{\mathrm{skel}})\). If a skeleton vertex \(v\) has
		\(\psi_{\cC_n^{\Root}}(v)\le\psi_{\cC_n^{\Root}}(v_{\mathrm{skel}})\), then
		\eqref{eq:graph-weighted-error} gives
		\[
		\psi_{\cB_n^{\Root}}^X(v)
		\le
		\psi_{\cC_n^{\Root}}(v)
		\le
		\psi_{\cC_n^{\Root}}(v_{\mathrm{skel}})
		\le
		\psi_{\cB_n^{\Root}}^X(v_{\mathrm{skel}})+\Delta_n
		\le
		\psi_{\cB_n^{\Root}}^X(v_{\mathrm{skel}})+\delta X(\cB_n^{\Root}).
		\]
		By \eqref{eq:graph-transfer-count}, there are at most \(L\) such skeleton
		vertices, including \(v_{\mathrm{skel}}\). Hence at most \(L\) vertices of
		\(\cC_n^{\Root}\) have graph-Jordan score at most
		\(\psi_{\cC_n^{\Root}}(v_{\mathrm{skel}})\), which implies \(v_{\mathrm{skel}}\in\TopJ_L(\cC_n^{\Root})\)
		for any deterministic tie rule.
	\end{proof}
	
	\begin{prop}[Root finding inside the root component]
		\label{prop:root-finding-root-component}
		Assume \eqref{eq:arch-q-range}. For every \(\eta>0\), there exists
		\(L=L(\eta,q,\gl)<\infty\) such that
		\begin{equation}
			\label{eq:root-component-jordan}
			\liminf_{n\to\infty}
			\PP\left\{
			1\in \TopJ_L(\cC_n^{\Root})
			\right\}
			\ge 1-\eta .
		\end{equation}
	\end{prop}
	
	\begin{proof}
		It is enough to consider \(0<\eta<1\).  Set \(\zeta=\eta/8\), choose the
		deterministic \(C_\zeta>1\) from
		Lemma~\ref{lem:decoration-moments}, and write
		\(E_n^{\mathrm{mom}}(C_\zeta)\) for the corresponding moment event.
		Let \(m=|\cB_n^{\Root}|\), and let \(\cF_n\) and \((\widetilde X_u)\) be the
		external sigma-field and raw weights defined before
		Lemma~\ref{lem:decoration-moments}.  The three conditional properties recorded
		when \(\cF_n\) was defined give precisely the environment needed below: the
		external component masses and \(m\) are \(\cF_n\)-measurable; after relabeling
		the root blob in birth order, its unrevealed genealogy is conditionally a
		uniform attachment tree; and the disjoint root-to-exterior shortcut families
		are conditionally independent of one another and of that genealogy.  Hence the
		raw weights are conditionally independent, positive, and independent of the
		skeleton.  On \(E_n^{\mathrm{mom}}(C_\zeta)\), their conditional moments satisfy
		the deterministic bounds \(c=1\) and \(C=C_\zeta\).
		
		Apply Lemma~\ref{lem:weighted-ua-root-finding}, in its uniform conditional
		form, with error \(\eta/4\), \(c=1\), and \(C=C_\zeta\), to
			the skeleton with weights \(\widetilde X\).  This gives finite
			\(L=L(\eta/4,1,C_\zeta)\),
			\(\delta=\delta(\eta/4,1,C_\zeta)>0\), and
			\(m_0=m_0(\eta/4,1,C_\zeta)<\infty\), none of which depends on \(n\) or
			on the realized environment.  Since \(\zeta=\eta/8\) and \(C_\zeta\) is
			deterministic once \((\eta,q,\gl)\) are fixed, these constants have only the
			claimed parameter dependence.  Let
		\(E_n^{\mathrm{UA}}\) be the event that
		\begin{equation}
			\label{eq:UA-root-gap}
			\psi_{\cB_n^{\Root}}^{\widetilde X}(1)
			\le
			(1-\delta)\widetilde X(\cB_n^{\Root})
		\end{equation}
		and
		\begin{equation}
			\label{eq:UA-competitor-count}
			\#\left\{
			v\in\cB_n^{\Root}:
			\psi_{\cB_n^{\Root}}^{\widetilde X}(v)
			\le
			\psi_{\cB_n^{\Root}}^{\widetilde X}(1)
			+\delta\widetilde X(\cB_n^{\Root})
			\right\}
			\le L.
		\end{equation}
		The application at the random size \(m\) is justified as
			follows.  The vertex set of \(\cB_n^{\Root}\), and hence \(m\), is
			\(\cF_n\)-measurable, and on \(E_n^{\mathrm{mom}}(C_\zeta)\) the skeleton and the raw
			weights form, conditionally on \(\cF_n\), an environment satisfying the
			hypotheses of Lemma~\ref{lem:weighted-ua-root-finding} at size \(m\), with
			the same deterministic moment bounds for every realized environment.
			Therefore, by
			\eqref{eq:uniform-conditional-ua}, on
			\(E_n^{\mathrm{mom}}(C_\zeta)\cap\{m\ge m_0\}\),
			\[
			\PP\left(
			(E_n^{\mathrm{UA}})^c
			\mid \cF_n
			\right)
			\le
			\frac{\eta}{4}.
			\]
			Taking expectations,
			\[
			\PP\big((E_n^{\mathrm{UA}})^c\big)
			\le
			\frac{\eta}{4}
			+\PP\big(E_n^{\mathrm{mom}}(C_\zeta)^c\big)
			+\PP(m<m_0) .
			\]
			By \eqref{eq:decoration-moments}, the upper limit of the middle failure
			probability is at most \(\zeta=\eta/8\).  Moreover,
			\(m=|\cB_n^{\Root}|\to\infty\) in probability, since
			\(n^{-q}|\cB_n^{\Root}|\) converges in distribution to a strictly positive
			limit by Proposition~\ref{prop:rrt-blob-input}, while
			\(m_0\) is a deterministic finite constant.  Consequently,
		\[
		\liminf_{n\to\infty}\PP(E_n^{\mathrm{UA}})
		\ge 1-\frac{\eta}{4}-\zeta
		=1-\frac{3\eta}{8}.
		\]
		
		Let
		\[
		E_n^{\mathrm{coup}}
		:=
		\left\{
		\widetilde X_u=X_u\text{ for all }u\in\cB_n^{\Root}
		\right\},
		\]
		and let \(E_n^{\mathrm{sing}}\) be the single-attachment event from
		Lemma~\ref{lem:singly-attached-decorations}.  Both events have probability
		tending to one.  On \(E_n^{\mathrm{sing}}\), put
		\[
		\Delta_n:=\max_{u\in\cB_n^{\Root}}|D_u|
		=
		\max_{u\in\cB_n^{\Root}}X_u .
		\]
		Set \(\Delta_n=0\) off \(E_n^{\mathrm{sing}}\), for definiteness.
		Since Lemma~\ref{lem:decoration-moments} gives
		\(\max_uX_u=o_{\PP}(m)\), while \(X(\cB_n^{\Root})\ge m\),
		\[
		E_n^{\Delta}
		:=
		E_n^{\mathrm{sing}}\cap
		\left\{
		\Delta_n\le \frac{\delta}{4}X(\cB_n^{\Root})
		\right\}
		\]
		has probability tending to one.
		
		On
		\[
		E_n^{\mathrm{UA}}\cap E_n^{\mathrm{coup}}\cap E_n^{\Delta},
		\]
		the following hold.  On \(E_n^{\mathrm{coup}}\) the weight
			vectors coincide, \(\widetilde X_u=X_u\) for all \(u\in\cB_n^{\Root}\), so
			\(\psi_{\cB_n^{\Root}}^{\widetilde X}\equiv\psi_{\cB_n^{\Root}}^{X}\) and
			\(\widetilde X(\cB_n^{\Root})=X(\cB_n^{\Root})\); hence the inequalities
			\eqref{eq:UA-root-gap} and \eqref{eq:UA-competitor-count} defining
			\(E_n^{\mathrm{UA}}\) hold verbatim with \(X\) in place of
			\(\widetilde X\).  These are exactly the second inequality of
			\eqref{eq:graph-transfer-gap} and the count \eqref{eq:graph-transfer-count},
			with \(v_{\mathrm{skel}}=1\) and the same \(\delta\) and \(L\), while
			\(E_n^{\Delta}\) supplies the first inequality of
			\eqref{eq:graph-transfer-gap}.  Thus the hypotheses of
			Lemma~\ref{lem:graph-to-weighted-jordan} hold with
			\(v_{\mathrm{skel}}=1\).  Therefore
		\(1\in\TopJ_L(\cC_n^{\Root})\) on this intersection.  The coupling,
		single-attachment, and maximal-decoration events have failure probability
		\(o(1)\).  Thus the lower limiting probability of this intersection is at
		least \(1-3\eta/8\), and hence the lower limit in
		\eqref{eq:root-component-jordan} is at least \(1-\eta\).
	\end{proof}
	
	\subsubsection{Proof of Theorem~\ref{thm:network-archeology-root-finding}}
	\label{sec:proof-network-archeology-root-finding}
	
	\begin{proof}
		Apply Theorem~\ref{thm:subcritical-blob-picture} with \(p=q\), which is
		allowed because \(q<1/(\gl+2)\).  Let
		\[
		A_{n,K}:=\{\rank(\cB_n^{\Root})\le K\}.
		\]
		By part (a), choose \(K\) so large that
		\[
		\liminf_{n\to\infty}\PP(A_{n,K})\ge 1-\eps/4 .
		\]
		For this fixed \(K\), let \(B_{n,K}\) be the event that the \(K\) largest full
		\(q\)-percolated components are exactly the components seeded by the \(K\)
		largest backbone blobs, with the deterministic tie rule.  By part (c),
		\(\PP(B_{n,K})\to1\).  On \(A_{n,K}\cap B_{n,K}\), the root blob is one of
		the \(K\) largest backbone blobs, and therefore the full component containing
		it is one of
		\[
		\cC_{n,1}^{(q)},\ldots,\cC_{n,K}^{(q)} .
		\]
		
		Choose \(L\) from Proposition~\ref{prop:root-finding-root-component} with
		\(\eta=\eps/2\), and put
		\[
		D_{n,L}:=
		\{1\in\TopJ_L(\cC_n^{\Root})\}.
		\]
		Then
		\[
		\liminf_{n\to\infty}\PP(D_{n,L})\ge 1-\eps/2 .
		\]
		On \(A_{n,K}\cap B_{n,K}\cap D_{n,L}\), the algorithm
		\(\cH_{K,L}\) examines the root component and includes the root among the
		\(L\) best Jordan vertices of that component.  Hence
		\[
		\PP\{1\in\cH_{K,L}(G_n^\circ)\}
		\ge
		\PP(A_{n,K}\cap B_{n,K}\cap D_{n,L}),
		\]
		and the union bound gives \eqref{eq:network-archeology-root-finding}.
	\end{proof}
	
	\section*{Acknowledgements}
	
	Shankar Bhamidi and Akshay Sakanaveeti were partially supported by NSF grants
	DMS-2413928 and DMS-2434559.  Shankar Bhamidi was also partially
	funded by NSF RTG grant DMS-2134107.
	
	\bibliographystyle{alpha}
	\bibliography{subcritical_blob_proof_refs}
	
\end{document}

%% file: style.tex
\usepackage{mathrsfs,xfrac} 
\usepackage[colorlinks=true,linkcolor=blue,citecolor=blue,urlcolor=blue]{hyperref} 
\usepackage{amsmath,amssymb,amsthm,amsfonts,amsbsy,latexsym,dsfont,color,graphicx,enumitem,tikz}
\usepackage[foot]{amsaddr}
\usepackage{caption}

\newtheorem{thm}{Theorem}[section]
\newtheorem{lem}[thm]{Lemma}
\newtheorem{cor}[thm]{Corollary}
\newtheorem{prop}[thm]{Proposition}

\theoremstyle{definition}
\newtheorem{rem}{Remark}[section]
\usepackage{etoolbox}
\AtEndEnvironment{rem}{\hfill\(\square\)}
\renewcommand{\le}{\leqslant}  
\renewcommand{\ge}{\geqslant} 
\newcommand{\eset}{\varnothing}

\newcommand{\ind}{\mathds{1}}
\newcommand{\eps}{\varepsilon}

\newcommand{\probc}{\stackrel{\mathrm{P}}{\longrightarrow}}
\newcommand{\weakc}{\stackrel{\mathrm{w}}{\longrightarrow}}

     \let\gl=\lambda

\newcommand{\cA}{\mathcal{A}}\newcommand{\cB}{\mathcal{B}}\newcommand{\cC}{\mathcal{C}}
\newcommand{\cD}{\mathcal{D}}\newcommand{\cF}{\mathcal{F}}
\newcommand{\cG}{\mathcal{G}}\newcommand{\cH}{\mathcal{H}}

\newcommand{\vw}{\mathbf{w}}
 
\newcommand{\mvw}{\boldsymbol{w}}

\newcommand{\w}{\mvw}
 
\usepackage{enumitem}

\newenvironment{enumeratei}
  {\begin{enumerate}[label=\textup{(\roman*)}, ref=\textup{(\roman*)}]}
  {\end{enumerate}}

\newenvironment{enumeratea}
  {\begin{enumerate}[label=\textup{(\alph*)}, ref=\textup{(\alph*)}]}
  {\end{enumerate}}

